\documentclass[a4paper,oneside]{amsart}
\usepackage[utf8]{inputenc}
\usepackage[a4paper,margin=3cm]{geometry} 
\usepackage{amsmath}
\usepackage{amsthm}
\usepackage{amscd}
\usepackage{amssymb}
\usepackage{latexsym}
\usepackage{eucal}
\usepackage{dsfont}
\usepackage[foot]{amsaddr}

\makeatletter
\renewcommand{\email}[2][]{%
  \ifx\emails\@empty\relax\else{\g@addto@macro\emails{,\space}}\fi%
  \@ifnotempty{#1}{\g@addto@macro\emails{\textrm{(#1)}\space}}%
  \g@addto@macro\emails{#2}%
}
\makeatother

\usepackage[colorlinks,pdftex]{hyperref}
\hypersetup{
linkcolor=black,
citecolor=black,
pdftitle={},
pdfauthor={Franziska K\"uhn, Ren\'e Schilling},
pdfkeywords={},
}

\renewcommand\labelenumi{\textup{(\roman{enumi})}}
\renewcommand\theenumi\labelenumi
\renewcommand\labelenumii{(\alph{enumii})}
\renewcommand\theenumii\labelenumii

\renewcommand\theenumii\labelenumii

\swapnumbers
\theoremstyle{plain}
\newtheorem{theorem}{Theorem}[section]
\newtheorem{lemma}[theorem]{Lemma}
\newtheorem{corollary}[theorem]{Corollary}
\newtheorem{proposition}[theorem]{Proposition}

\theoremstyle{definition}
\newtheorem{remark}[theorem]{Remark}
\newtheorem{example}[theorem]{Example}

\DeclareMathOperator \re {Re}

\DeclareMathOperator \dv {div}

\newcommand\mc[1] {\mathcal{#1}}
\newcommand\mbb[1] {\mathds{#1}}

\newcommand{\eps}{\epsilon}
\newcommand{\I}{\mathds{1}}

\newcommand\fa{\quad\text{for all \ }}
\newcommand{\firstpara}[1]{ \ \textup{#1\ \ }}
\newcommand{\para}[1]{\bigskip\textup{#1\ \ }}

\begin{document}

\title[Euler--Maruyama approximation for L\'evy-driven SDEs]{Strong convergence of the Euler--Maruyama approximation for a class of L\'evy-driven SDEs}

\author[F.~K\"{u}hn]{Franziska K\"{u}hn}
\address[F.~K\"{u}hn]{TU Dresden\\ Fachrichtung Mathematik\\ Institut f\"{u}r Mathematische Stochastik\\ 01062 Dresden, Germany. E-Mail: \textnormal{franziska.kuehn1@tu-dresden.de}.}

\author[R.\,L.~Schilling]{Ren\'e L.\ Schilling}
\address[R.\,L.~Schilling]{TU Dresden\\ Fachrichtung Mathematik\\ Institut f\"{u}r Mathematische Stochastik\\ 01062 Dresden, Germany. E-Mail: \textnormal{rene.schilling@tu-dresden.de}.}


\subjclass[2010]{Primary: 60H10. Secondary: 60G51, 60J35, 60J75.}
\keywords{Euler--Maruyama approximation, stochastic differential equation, strong convergence.}

\begin{abstract}
    Consider the following stochastic differential equation (SDE)
    \begin{equation*}
    	dX_t = b(t,X_{t-}) \, dt+ dL_t, \quad X_0 = x,
    \end{equation*}
    driven by a $d$-dimensional L\'evy process $(L_t)_{t \geq 0}$. We establish conditions on the L\'evy process and the drift coefficient $b$ such that the Euler--Maruyama approximation converges strongly to a solution of the SDE with an explicitly given rate. The convergence rate depends on the regularity of $b$ and the behaviour of the L\'evy measure at the origin. As a by-product of the proof, we obtain that the SDE has a pathwise unique solution. Our result covers many important examples of L\'evy processes, e.g.\ isotropic stable, relativistic stable, tempered stable and layered stable.
\end{abstract}

\maketitle

\section{Introduction}

For a given  L\'evy process $(L_t)_{t \geq 0}$ with values in $\mbb{R}^d$ and L\'evy triplet $(\ell,Q,\nu)$ we consider the stochastic differential equation (SDE)
\begin{equation}\label{intro-eq1}
	dX_t = b(t,X_{t-}) \, dt + dL_t, \quad X_0 = x\in\mbb{R}^d.
\end{equation}
If the drift coefficient $b$ is H\"{o}lder continuous in time and space, there is a quite general result on the existence of a pathwise unique solution, cf.\ Chen, Song \& Zhang \cite{chen}. It is, however, in general not possible to calculate the solution explicitly, and therefore it is important to have numerical schemes which allow us to approximate the solution. In this paper we derive conditions on the L\'evy process $(L_t)_{t \geq 0}$ and the drift coefficient $b$ such that the Euler--Maruyama approximation
\begin{equation*}
	X_t^{(n)} -x = \int_0^t b \big(s,X_{\eta_n(s)-}^{(n)} \big) + L_t,
    \quad t \in [0,T],\; n\in\mbb{N},
\end{equation*}
converges strongly to a solution of the given SDE with a certain rate; here $\eta_n(s) := T\frac{i}{n}$ for $s \in [T \frac{i}{n}, T \frac{i+1}{n})$. It turns out that the convergence (rate) depends on two factors: the regularity of $b$ and the behaviour of the L\'evy measure at the origin.

If $b=b(x)$ satisfies a one-sided Lipschitz condition, then a result by Higham \& Kloeden \cite{kloeden} shows that the Euler--Maruyama approximation converges strongly with convergence rate $1/2$. It is natural to ask whether the regularity assumption can be weakened to H\"{o}lder regularity. Pamen \& Taguchi \cite{pamen} study the convergence rate for SDEs with H\"{o}lder continuous coefficients driven by Brownian motion and by truncated $\alpha$-stable L\'evy processes with index $\alpha>1$. For isotropic $\alpha$-stable L\'evy processes, $\alpha>1$, first results were obtained by Qiao \cite{qiao} ($b$ is Lipschitz up to a $\log$-term) and by Hashimoto \cite{hashimoto} who proves the strong convergence under a Komatsu condition, but does not determine the convergence rate. More recently, Mikulevicius \& Xu \cite{mik16} have shown that
\begin{equation*}
	\mbb{E} \left( \sup_{0 \leq t \leq T} |X_t-X_t^{(n)}|^p \right)
    \leq C {n^{-p \beta/\alpha}}, \quad p \in (0,\alpha),
\end{equation*}
if $\alpha \in (1,2)$ and $b$ is $\beta$-H\"{o}lder continuous for some $\beta>1-\alpha/2$.  This estimate shows, in particular, that there are two factors which result in a slow convergence of the Euler--Maruyama approximation: weak regularity of $b$ and a strong singularity of the L\'evy measure $\nu(dy) = |y|^{-(d+\alpha)} \, dy$ at $y=0$. The fact that the behaviour of the L\'evy measure influences the convergence rate was already observed by Jacod \cite{jacod} who investigated the weak convergence of the Euler--Maruyama approximation for a class of L\'evy-driven SDEs.

Our main result, Theorem~\ref{main1}, shows the strong convergence of the Euler--Maruyama approximation for a large class of driving L\'evy processes covering many important and interesting examples, e.g.\ isotropic $\alpha$-stable, relativistic stable, tempered stable and layered stable L\'evy processes. The proof relies on the so-called It\^o-Tanaka trick which relates the time average $\int_0^t b(s,X_s) \, ds$ of the solution $(X_t)_{t \geq 0}$ to \eqref{intro-eq1} with the solution $u$ to the Kolmogorov equation
\begin{equation}\label{intro-eq7}
	\partial_t u(t,x) + A_x u(t,x) + b(t,x) \cdot \nabla_x u(t,x) = - b(t,x);
\end{equation}
here $A_x$ denotes the infinitesimal generator of the driving L\'evy process $(L_t)_{t \geq 0}$ acting with respect to the space variable $x$. The key step is to prove the existence of a solution to \eqref{intro-eq7} which is sufficiently regular and satisfies certain H\"{o}lder estimates. The required regularity of $u$ depends on the regularity of $b$ and the behaviour of the L\'evy measure at $0$.

The It\^o-Tanaka trick has been used by Pamen \& Taguchi \cite{pamen} to prove the strong convergence of the Euler--Maruyama approximation for the particular case that $(L_t)_{t \geq 0}$ is a Brownian motion or a truncated stable L\'evy process taking values in $\mbb{R}^d$, $d \geq 2$. For these two processes the existence of a sufficiently nice solution to the Kolmogorov equation \eqref{intro-eq7} was already known. Pamen \& Taguchi do not take advantage of the fact that the required regularity of the solution depends on the behaviour of the L\'evy measure at $0$, and therefore they end up with a convergence rate which is far from being optimal.

\medskip
Our paper is organized as follows. In Section~\ref{main} we state and discuss the main results; the required definitions will be explained in Section~\ref{pre}. Section~\ref{p} is devoted to the proofs of the main results, and in Section~\ref{ex} we illustrate our results with examples. Some auxiliary statements are proved in the appendix.

\section{Main results} \label{main}

Throughout $(L_t)_{t \geq 0}$ is a $d$-dimensional L\'evy process with characteristic function $\mbb{E}e^{i \xi \cdot L_t}$. It is well-known that $\mbb{E}e^{i \xi \cdot L_t} = e^{-t \psi(\xi)}$ with a characteristic exponent $\psi$ which is uniquely characterized via the L\'evy--Khintchine formula by the L\'evy triplet $(\ell,Q,\nu)$ consisting of $\ell \in \mbb{R}^d$, $Q \in \mbb{R}^{d \times d}$ positive semidefinite and a measure $\nu$ on $\mbb{R}^d \backslash \{0\}$ such that $\int \min\{1,|y|^2\} \, \nu(dy)<\infty$, see \eqref{pre-eq1} and \eqref{pre-eq3} p.~\pageref{pre-eq1}.

\begin{theorem} \label{main1}
    Let $(L_t)_{t \geq 0}$ be a $d$-dimensional L\'evy process with L\'evy triplet $(\ell,0,\nu)$ and characteristic exponent $\psi: \mbb{R}^d \to \mbb{C}$, and let $\gamma_0 \in [1,2]$, $\gamma_{\infty}>0$ be such that $\int_{|z| \leq 1} |z|^{\gamma_0} \, \nu(dz) < \infty$ and $\int_{|z| \geq 1} |z|^{\gamma_{\infty}} \, \nu(dz)<\infty$. Assume that $L_t$ admits a transition density $p_t \in C^2(\mbb{R}^d)$, i.\,e.\ $x \mapsto p_t(x)$ is twice continuously differentiable, for all $t>0$, such that
	 there exist constants $\alpha \in (1,2]$ and $c=c(T)>0$ such that
    \begin{equation}\label{ihke}
		\int_{\mbb{R}^d} |\partial_{x_i} p_t(x)| \, dx \leq c t^{-1/\alpha}
        \fa  i \in \{1,\dots,d\},\; t \in (0,T].
	\end{equation}
    Let $b: [0,\infty) \times \mbb{R}^d \to \mbb{R}^d$ be a bounded function which is $\beta$-H\"{o}lder continuous with respect to $x$ and $\eta$-H\"{o}lder continuous with respect to $t$ for some $\beta,\eta \in (0,1]$, i.e.\
    \begin{equation*}
		 |b(t,x)-b(t,y)| \leq C |x-y|^{\beta}
    \quad\text{and}\quad
        |b(s,x)-b(t,x)| \leq C |s-t|^{\eta}
	\end{equation*}
	holds for all $s,t \geq 0$, $x,y \in \mbb{R}^d$ and with an absolute constant $C>0$. If the balance condition
    \begin{gather}\label{bc}
		2 \alpha - \gamma_0 (1-\beta)>2,
	\intertext{is satisfied, then the SDE }\label{sde}
		dX_t = b(t,X_{t-}) \, dt + dL_t, \quad X_0 = x
    \intertext{has a pathwise unique strong solution $(X_t)_{t \geq 0}$, and for any $p \leq \gamma_{\infty}$ and $T>0$ there exists a constant $C>0$ such that}\label{ineq}
		\mbb{E} \left( \sup_{0 \leq t \leq T} |X_t-X_t^{(n)}|^p \right)
        \leq C {n^{-\min\{1,\,p\beta/\gamma_0,\,p \eta\}}} \fa n \in \mbb{N}.
	\end{gather}
\end{theorem}

\begin{remark} \label{main5}
\firstpara{(i)}
	For the existence of a pathwise unique solution to \eqref{sde} it is crucial that the mapping $x \mapsto b(t,x)$ is sufficiently regular. For instance, if $(L_t)_{t \geq 0}$ is an isotropic $\alpha$-stable L\'evy process, then the SDE \begin{equation*}
		dX_t = b(X_{t-}) \, dt+ dL_t
	\end{equation*}
	fails, in general, to have a pathwise unique solution if $b$ is $\beta$-H\"older continuous with $\beta+\alpha<1$, cf.\ \cite{tanaka}; recently, Kulik \cite{kulik} has shown that the SDE admits a pathwise unique solution if $\beta+\alpha>1$. This shows that there is a trade-off and compensation between the (lack of) regularity of the driving L\'evy noise and the (lack of) regularity of the coefficient $x\mapsto b(x)$.

\para{(ii)}
    Condition \eqref{ihke} is equivalent to saying that the semigroup $P_t \phi(x) := \mbb{E}\phi(x+L_t)$, $\phi \in \mc{B}_b(\mbb{R}^d)$, associated with the L\'evy process $(L_t)_{t \geq 0}$ satisfies the gradient estimate
    \begin{equation*}
    	\|\nabla P_t \phi\|_{\infty}
        \leq ct^{-1/\alpha} \|\phi\|_{\infty},
        \quad \phi \in \mc{B}_b(\mbb{R}^d),
    \end{equation*}
    cf.\ Lemma~\ref{p-0}. If $(L_t)_{t \geq 0}$ is subordinate to a Brownian motion, then \eqref{ihke} can also be understood as a moment estimate, cf.\ Lemma~\ref{p-7}.

\para{(iii)}
    The existence of the moments $\int_{|z| \leq 1} |z|^{\gamma_0} \, \nu(dz)$ and $\int_{|z| \geq 1} |z|^{\gamma_{\infty}} \, \nu(dz)$ is related to the growth of the characteristic exponent $\psi$, cf.\ Lemma~\ref{ex-1};  Lemma~\ref{ex-1} is very useful since it allows us to verify the assumptions of Theorem~\ref{main1} if the L\'evy measure $\nu$ cannot be calculated explicitly.

\para{(iv)}
    A sufficient condition for the existence of a transition probability density $p_t \in C^2(\mbb{R}^d)$ for all $t>0$ is the Hartman--Wintner condition: \begin{equation*}
		\lim_{|\xi| \to \infty} \frac{\re \psi(\xi)}{\log(1+|\xi|)} = \infty,
	\end{equation*}
	cf.\ \cite{knop} for a thorough discussion.

\para{(v)}
    Let $(X_t)_{t \geq 0}$ be a solution to \eqref{sde}. Since $b$ is bounded, we have for any $t>0$
    \begin{equation*}
		X_t \in L^p(\mbb{P})
        \iff L_t \in L^p(\mbb{P})
        \iff \int_{|z| \geq 1} |z|^p \, \nu(dz) < \infty
	\end{equation*}
    -- for the second equivalence see Sato \cite{sato} --, i.e.\ the solution inherits the integrability of the driving L\'evy process and vice versa. This means that, in general, we cannot expect \eqref{ineq} to hold for $p>\gamma_{\infty}$.
   
   \para{(vi)}
	   A slight variation of our arguments, see the uniqueness part of the proof of Theorem~\ref{main1} on page \pageref{p-eq37}, allows us to derive the estimate \begin{equation*}
		   \mbb{E} \left( \sup_{t \in [0,T]} |X_t(x)-X_t(y)|^p \right) \leq C' |x-y|^p, \qquad p \leq \gamma_{\infty}
	  \end{equation*}
	   for some constant $C'>0$ where $X_t(z)$ denotes the solution to the SDE with initial condition $X_0(z)=z$.  If $\gamma_{\infty}>d$ it follows from a standard Kolmogorov--Chentsov--Totoki argument that $x\mapsto X_t(x)$ is H\"older continuous of order $\kappa<1-d/\gamma_{\infty}$.
\end{remark}

Let us give some further remarks on possible extensions of Theorem~\ref{main1}.

\begin{remark}
\firstpara{(i)}
    If $(L_t)_{t \geq 0}$ has a non-vanishing (possibly degenerate) diffusion part, then the statement of Theorem~\ref{main1} remains valid for $\gamma_0 := 2$. In particular, if $(W_t)_{t \geq 0}$ is a Brownian motion and $b: [0,\infty) \times \mbb{R}^d \to \mbb{R}^d$ is a bounded function which is $\beta$-H\"{o}lder continuous with respect to $x$ and $\eta$-H\"{o}lder-continuous with respect to $t$ for some $\beta,\eta \in (0,1]$, then the SDE
    \begin{gather*}
    	dX_t = b(t,X_t) \, dt + dW_t, \quad X_0 = x
    \intertext{has a pathwise unique solution, and for any $p>0$, $T>0$ there exists a constant $C>0$ such that}
    	\mbb{E} \left( \sup_{0 \leq t \leq T} |X_t-X_t^{(n)}|^p \right) \leq C {n^{-\min\{1,\,p\beta/2,\, p\eta\}}} \fa n \in \mbb{N};
    \end{gather*}
    this result extends \cite[Theorem 2.11]{pamen}.

   \para{(ii)}  A close inspection of our arguments reveals that the H\"{o}lder condition on $t\mapsto b(t,x)$ can be replaced by uniform continuity; if we denote by \begin{equation*}
	   w(\delta) := \sup_{x \in \mbb{R}^d} \sup_{|s-t| \leq \delta} |b(t,x)-b(s,x)|
	  \end{equation*}
	  the modulus of continuity of $t \mapsto b(t,x)$ (uniformly in $x$), then \eqref{ineq} becomes \begin{equation*}
		  \mbb{E} \left( \sup_{0 \leq t \leq T} |X_t-X_t^{(n)}|^p \right) \leq C n^{-\min\{1,p \beta/\gamma_0\}} + C w(1/n)^p, \qquad n \in \mbb{N},
	 \end{equation*}
	for a suitable constant $C>0$.
	
  \para{(iii)} Case $\alpha=1$: If we replace \eqref{ihke} by  \begin{equation*}
	  \int_{\mbb{R}^d} |\partial_{x_i} p_t(x)| \, dx 
        \leq c t^{-1} \log^{-1-\eps} t, \quad i \in \{1,\ldots,d\}, \; t \in (0,T]
	 \end{equation*}
	 for some $\eps>0$, then the statement of Theorem~\ref{main1} holds for $\alpha=1$.
 \end{remark}
 
 \begin{remark}
	 It is a natural question whether it is possible to consider \eqref{intro-eq1} with multiplicative noise, i.\,e.\ with $\sigma(t,X_{t-}) \, dL_t$ instead of $dL_t$. This, however, requires different methods than the ones used here, see the recent paper \cite{chen2} by Chen et.\ al; in particular, it is not clear how to find a solution to the PDE \eqref{p-eq21} appearing in Theorem~\ref{p-5}. 
\end{remark}

Combining Theorem~\ref{main1} with the gradient estimates in \cite{ssw} we can easily prove the following statement which covers many interesting and important examples of L\'evy processes, cf.\ Section~\ref{ex}.
\begin{corollary} \label{main3}
    Let $(L_t)_{t \geq 0}$ be a $d$-dimensional L\'evy process with characteristic exponent $\psi$ and L\'evy triplet $(\ell,0,\nu)$.  Let $\gamma_0 \in [1,2]$, $\gamma_{\infty}>0$ be exponents such that $\int_{|z| \leq 1} |z|^{\gamma_0} \, \nu(dz) < \infty$ and $\int_{|z| \geq 1} |z|^{\gamma_{\infty}} \, \nu(dz)<\infty$. Assume that there exists a strictly increasing function $f: (0,\infty) \to [0,\infty)$ which is differentiable near infinity and satisfies the following conditions.
    \begin{enumerate}
		\item $c^{-1}f(|\xi|) \leq \re \psi(\xi) \leq c f(|\xi|)$ as $|\xi| \to \infty$ for some constant $c\in(0,\infty)$;
		\item $\limsup_{r \to \infty} f^{-1}(2r)/f^{-1}(r) < \infty$;
		\item there exist constants $\alpha \in (1,2]$ and $c>0$ such that $f(r) \geq c r^{\alpha}$ for large $r>0$.
	\end{enumerate}
    Let $b: [0,\infty) \times \mbb{R}^d \to \mbb{R}^d$ be a bounded function which is $\beta$-H\"{o}lder continuous with respect to $x$ and $\eta$-H\"{o}lder continuous with respect to $t$ for some $\beta,\eta \in (0,1]$. If the balance condition holds
    \begin{gather*}
		2 \alpha-\gamma_0 (1-\beta)>2,
	\intertext{then the SDE}
			dX_t = b(t,X_{t-}) \, dt + dL_t, \quad X_0 = x
    \intertext{has a pathwise unique strong solution $(X_t)_{t \geq 0}$, and for any $p \leq \gamma_{\infty}$ and $T>0$ there exists a constant $C>0$ such that}
		\mbb{E} \left( \sup_{0 \leq t \leq T} |X_t-X_t^{(n)}|^p \right)
        \leq C {n^{-\min\{1,\,p\beta/\gamma_0,\,p\eta\}}} \fa n \in \mbb{N}.
	\end{gather*}
\end{corollary}

Typical examples for $f$ are $f(r) = r^{\alpha}$ and $f(r) = r^{\alpha} \log^\beta(1+r)$, see Section~\ref{ex}.

\medskip
For the particular case that the driving L\'evy process is subordinate to a Brownian motion, Theorem~\ref{main1} has the following corollary.
\begin{corollary} \label{main2}
    Let $L_t = B_{S_t}$ be a $d$-dimensional Brownian motion subordinated by a subordinator $(S_t)_{t \geq 0}$ with Laplace exponent \textup{(}Bernstein function\textup{)} $f$,
    \begin{equation*}
		f(\lambda) = \int_{(0,\infty)}(1-e^{-\lambda r}) \, \mu(dr), \quad \lambda \geq 0.
	\end{equation*}
	Assume that there exist constants $\delta_0 \in [1/2,1]$, $\delta_{\infty} >0$ and $\rho \in (1/2,1]$ such that
    \begin{gather*}
		\smash[b]{\int_{(0,1)} r^{\delta_0} \, \mu(dr) + \int_{(1,\infty)} r^{\delta_{\infty}} \, \mu(dr)< \infty}
	\intertext{and}
		\smash[t]{\liminf_{\lambda \to \infty} \frac{f(\lambda)}{\lambda^{\rho}}>0.}
	\end{gather*}
    Let $b: [0,\infty) \times \mbb{R}^d \to \mbb{R}^d$ be a bounded function which is $\beta$-H\"{o}lder continuous with respect to $x$ and $\eta$-H\"{o}lder continuous with respect to $t$ for some $\eta \in (0,1]$.  If
    \begin{gather*}
		2 \rho- \delta_0 (1-\beta)>1
	\intertext{then the SDE}
		dX_t = b(t,X_{t-}) \, dt + dL_t, \quad X_0 = x
	\intertext{has a unique strong solution $(X_t)_{t \geq 0}$ and for any $p \leq 2\delta_{\infty}$ and $T>0$ there exists a constant $C>0$ such that}
		\mbb{E} \left( \sup_{0 \leq t \leq T} |X_t-X_t^{(n)}|^p \right) \leq C {n^{-\min\{1,\,p\beta/(2\delta_0),\,p\eta\}}} \fa n \in \mbb{N}.
	\end{gather*}
\end{corollary}

\section{Preliminaries} \label{pre}

We consider Euclidean space $\mbb{R}^d$ endowed with the Borel $\sigma$-algebra $\mathcal{B}(\mbb{R}^d)$. The open ball centered at $x \in \mbb{R}^d$ of radius $r>0$ is denoted by $B(x,r)$. For a differentiable function $f: \mbb{R}^d \to \mbb{R}$ the partial derivative with respect to $x_i$ is denoted by $\partial_{x_i} f$, and $\nabla f$ is the gradient of $f$. As usual, $C^k(\mbb{R}^d)$ is the space of $k$-times continuously differentiable functions, $\mc{B}_b(\mbb{R}^d)$ the space of bounded Borel measurable functions, and $C_{\infty}(\mbb{R}^d)$ is the space of continuous functions vanishing at infinity. For $\beta \in [0,1]$ we define H\"{o}lder spaces by
\begin{align*}
    \mc{C}_b^{\beta}(\mbb{R}^d)
    &:= \left\{f: \mbb{R}^d \to \mbb{R}^k;\; \|f\|_{\mc{C}_b^{\beta}(\mbb{R}^d)} := \sup_{x \in \mbb{R}^d} |f(x)| +  \smash{\sup_{x \neq y} \frac{|f(x)-f(y)|}{|x-y|^{\beta}}} < \infty \right\}
    \\
	\mc{C}_b^{1,\beta}(\mbb{R}^d)
    &:= \left\{f \in C^1(\mbb{R}^d,\mbb{R}^k);\; \|f\|_{\mc{C}_b^{1,\beta}(\mbb{R}^d)} := \sup_{x \in \mbb{R}^d} |f(x)| + \|\nabla f\|_{\mc{C}_b^{\beta}(\mbb{R}^d)}< \infty \right\}.
\end{align*}
For a function space $M$ and a function $g: [a,b] \times \mbb{R}^d \to \mbb{R}^k$ we write $g \in C([a,b],M)$ if $g(t,\cdot) \in M$ for all $t \in [a,b]$ and $t \mapsto g(t,\cdot)$ is continuous. Similarly, $g \in C^1([a,b],M)$ means that $\partial_t g(t,\cdot) \in M$ and $t \mapsto \partial_t g(t,\cdot)$ is continuous. If $M$ is a normed function space, then
\begin{equation*}
	\|g\|_{C([a,b],M)} := \sup_{t \in [a,b]} \|g(t,\cdot)\|_M
\end{equation*}
defines a norm on $C([a,b],M)$. For brevity we will often denote this norm by $\|g\|_M$; in particular we will write
\begin{equation*}
    \|g\|_{\mc{C}_b^{\beta}(\mbb{R}^d)}
    := \sup_{t \in [a,b]} \sup_{x \in \mbb{R}^d} |g(t,x)| + \sup_{t \in [a,b]} \sup_{x \neq y} \frac{|g(t,x)-g(t,y)|}{|x-y|^{\beta}}.
\end{equation*}
We say that a bounded function $g:[a,b] \times \mbb{R}^d \to \mbb{R}^k$ is \emph{$\beta$-H\"{o}lder continuous with respect to $x$} if $\|g\|_{\mc{C}_b^{\beta}(\mbb{R}^d)}<\infty$. \par	

Throughout, $(\Omega,\mc{A},\mbb{P})$ is a probability space. A family of random variables $L_t: \Omega \to \mbb{R}^d$, $t \geq 0$, is a \emph{$d$-dimensional L\'evy process} if $(L_t)_{t \geq 0}$ has stationary and independent increments, $t \mapsto L_t$ is, with probability $1$, right-continuous with finite left limits (c\`adl\`ag), and $L_0 = 0$. A L\'evy process can be uniquely (in the sense of finite-dimensional distributions) characterized by its \emph{characteristic exponent} $\psi: \mbb{R}^d \to \mbb{C}$,
\begin{equation}
	\mbb{E}e^{i \xi \cdot L_t} = e^{-t \psi(\xi)}, \quad t  \geq 0,\; \xi \in \mbb{R}^d; \label{pre-eq1}
\end{equation}
the exponent is given by the \emph{L\'evy--Khintchine formula}
\begin{equation}
    \psi(\xi)
    = -i \ell \cdot \xi + \frac{1}{2} \xi \cdot Q \xi + \int_{y \neq 0} \left(1-e^{iy \cdot \xi} + iy \cdot \xi \I_{(0,1)}(|y|)\right) \nu(dy) ,
    \quad \xi \in \mbb{R}^d. \label{pre-eq3}
\end{equation}
There is a one-to-one correspondence between the exponent $\psi$ and the \emph{L\'evy triplet} $(\ell,Q,\nu)$ consisting of a vector $\ell \in \mbb{R}^d$ (\emph{drift parameter}), a symmetric positive semi-definite matrix $Q \in \mbb{R}^{d \times d}$ (\emph{diffusion parameter}) and a measure $\nu$ on $(\mbb{R}^d \setminus \{0\},\mc{B}(\mbb{R}^d \setminus \{0\}))$ satisfying $\int_{y \neq 0} \min\{|y|^2,1\} \, \nu(dy)<\infty$ (\emph{L\'evy measure}).

It is not difficult to see that any L\'evy process $(L_t)_{t \geq 0}$ is a Markov process, and therefore there is a transition semigroup $(P_t)_{t \geq 0}$ and an infinitesimal generator $(A,\mc{D}(A))$ associated with $(L_t)_{t \geq 0}$. It is well known that
\begin{gather}
\notag
	P_t f(x) = \mbb{E}f(x+L_t),
    \quad x \in \mbb{R}^d,\; t \geq 0,\; f \in \mc{B}_b(\mbb{R}^d),
\intertext{and, for any $f \in C_b^2(\mbb{R}^d)$,}
\label{pre-eq5}
    Af(x)
    = \ell \cdot \nabla f(x) + \frac{1}{2} \dv\big(Q \nabla f(x)\big) + \int_{y \neq 0} \left( f(x+y)-f(x)-\nabla f(x) \cdot y \I_{(0,1)}(|y|) \right) \nu(dy).
\end{gather}
If $(L_t)_{t \geq 0}$ is a L\'evy process with L\'evy triplet $(\ell,0,\nu)$ and $\int_{|y| \leq 1} |y|^{\gamma} \, \nu(dy) < \infty$ for some $\gamma \in [1,2]$, then $\mc{C}_b^{1,\gamma-1}(\mbb{R}^d) \subseteq \mc{D}(A)$, and \eqref{pre-eq5} holds for any $f \in \mc{C}_b^{1,\gamma-1}(\mbb{R}^d)$, cf.\ \cite[Theorem 4.1]{ihke}; moreover,
\begin{equation}\label{pre-eq6}
    \|Af\|_{\infty}
    \leq 2 \left( |\ell| + \int_{y \neq 0} \min\{|y|^{\gamma},1\} \, \nu(dy) \right)  \|f\|_{\mc{C}_b^{1,\gamma-1}(\mbb{R}^d)}
    \fa f \in \mc{C}_b^{1,\gamma-1}(\mbb{R}^d).
\end{equation}
For a function $f=f(t,x)$ we indicate by $A_x f(t,x)$ that the operator $A$ acts on the variable $x$ for any fixed $t$. Our standard reference for L\'evy processes is the monograph \cite{sato} by Sato.

A L\'evy process $(S_t)_{t\geq 0}$ with non-decreasing sample paths is called a \emph{subordinator}. It can be uniquely characterized by its Laplace exponent (Bernstein function) $f(r) := \log \mbb{E} e^{- rS_t}$,
\begin{equation*}
	f(r) = br + \int_{(0,\infty)} (1-e^{\lambda r}) \, \mu(dr), \quad r>0,
\end{equation*}
where $b \geq 0$ and $\mu$ is a measure on $(0,\infty)$ such that $\int_{(0,\infty)} \min\{r,1\} \, \mu(dr)<\infty$. If $(L_t)_{t \geq 0}$ is a L\'evy process with characteristic exponent $\psi$ and $(S_t)_{t \geq 0}$ a subordinator with Laplace exponent $f$ such that $(L_t)_{t \geq 0}$ and $(S_t)_{t \geq 0}$ are independent, then the subordinate process
\begin{equation*}
	X_t(\omega) := L_{S_t(\omega)}(\omega), \quad \omega \in \Omega,\; t \geq 0,
\end{equation*}
is a L\'evy process and its characteristic exponent is given by $f(\psi(\xi))$, see \cite{bernstein} for further details.

A stochastic differential equation (SDE) driven by a L\'evy process is of the form
\begin{equation*}
    dX_t = b(t,X_{t-}) \, dt + g(t,X_{t-}) \, dL_t
\end{equation*}
for suitable coefficients $b$, $g$ and an initial condition fixing $X_0$. The \emph{Euler--Maruyama approximation} of the SDE is given by \begin{equation}\label{pre-eq9}
	X_t^{(n)} := X_0 + \int_0^t b(s,X_{\eta_n(s)-}^{(n)}) \, ds + \int_0^t g(s,X_{\eta_n(s)-}^{(n)}) \, dL_s,
    \quad t \in [0,T]
\end{equation}
where $\eta_n(s) :=T \frac{i}{n}$ for any $s \in [T\frac{i}{n},T\frac{i+1}{n})$, $i =0,1,\dots, n$, for fixed $T>0$. We say that the SDE has a \emph{pathwise unique} solution if for any two solutions $(X_t)_{t \geq 0}$ and $(Y_t)_{t \geq 0}$ such that $X_0 = Y_0$, we have
\begin{equation*}
	\mbb{P} \left( \forall t \geq 0: X_t = Y_t \right)=1.
\end{equation*}
We refer to Ikeda--Watanabe \cite{ikeda} and Protter \cite{protter} for a thorough discussion of stochastic integration and SDEs. A key tool for the proof of Theorem~\ref{main1} is Novikov's \cite{novikov} version of the Burkholder--Davis--Gundy inequality for jump processes which holds for \emph{any} $p>0$. Let us state this result from \cite{novikov} in our notation.

\begin{theorem}[Novikov 1975] \label{pre-1}
	Let $\tilde{N}$ be a Poisson random measure with compensator $\hat{N}$ of the form $\hat{N}(dy,ds) = \nu(dy) \, ds$, and let $H$ be a predictable stochastic process. \begin{enumerate}
		\item\label{pre-1-i} If $\mathbb{E}\left( \int_0^T \!\! \int_{y \neq 0} |H(s,y)| \, \nu(dy) \, ds \right)<\infty$ then \begin{equation*}
			\mbb{E} \left( \sup_{t \leq T} \left| \int_0^t\!\! \int H(s,y) \tilde{N}(dy,ds) \right|^p \right) 
			\leq C_{p,\alpha} \mbb{E} \left[ \left( \int_0^T \!\! \int |H(s,y)|^{\alpha} \, \nu(dy) \, ds \right)^{p/\alpha} \right]
		\end{equation*}
		for any $\alpha \in [1,2]$ and $p \in [0,\alpha]$; here $C_{p,\alpha}$ is an absolute constant depending only on $\alpha$ and $p$.
		\item\label{pre-1-ii} If $\mbb{E} \left( \int_0^T \!\! \int_{y \neq 0} |H(s,y)|^2 \, \nu(dy) \, ds \right)<\infty$ and $p \geq 2$, then there exists an absolute constant $c_p>0$ such that \begin{align*}
			\mbb{E} \left( \sup_{t \leq T} \left| \int_0^t \!\! \int H(s,y) \, \tilde{N}(dy,ds) \right|^p \right)
			&\leq c_p \mbb{E} \left[\left( \int_0^T \!\! \int |H(s,y)|^2 \, \nu(dy) \, ds \right)^{p/2} \right] \\
			&\qquad+ c_p \mbb{E} \left( \int_0^T \!\! \int |H(s,y)|^p \, \nu(dy) \, ds \right).
		\end{align*} 
	\end{enumerate}
\end{theorem}

\section{Proofs} \label{p}
Before we start to prove Theorem~\ref{main1} let us briefly explain the idea of the proof in dimension $d=1$; the general case $d>1$ follows from this by considering the coordinates. Suppose that $(X_t)_{t \geq 0}$ solves the SDE
\begin{equation*}
	dX_t = b(t,X_{t-}) \, dt + dL_t, \quad X_0 = x.
\end{equation*}
From the definition of the Euler--Maruyama approximation \eqref{pre-eq9} we see that
\begin{equation}\label{p-eq3}
	X_t-X_t^{(n)} = \int_0^t \left(b(s,X_{s-})-b(s,X_{\eta_n(s)-}^{(n)}) \right) ds,
\end{equation}
and so we have to show that the right-hand side converges in $L^p(\mbb{P})$ to $0$ as $n \to \infty$. To this end, we use the so-called It\^o-Tanaka trick: We will show that there exists a sufficiently well-behaved solution $u: [0,T] \times \mbb{R}^d \to \mbb{R}$ to the integro-differential equation \begin{equation}\label{p-eq5}
	\frac{\partial}{\partial t} u(t,x) + A_x u(t,x) + b(t,x) \cdot \nabla_x u(t,x) = - b(t,x),
    \quad u(T,x)=0
\end{equation}
for small $T>0$; by $A$ we denote the generator of the driving L\'evy process $(L_t)_{t \geq 0}$. Applying It\^o's formula, we get
\begin{align*}
	\int_0^t (b(s,X_{s-})-b(s,X_{\eta_n(s)-}^{(n)}) \, ds \approx u(t,X_t)-u(t,X_t^{(n)}) + M_t
\end{align*}
for some martingale $M$. If $u$ is sufficiently smooth, this will allow us to estimate the $L^p$-norm of right-hand side of \eqref{p-eq3} using the Burkholder--Davis--Gundy inequality, see pp.~\pageref{p-eq27}.

\medskip
In the first part of this section we establish the existence of a solution to \eqref{p-eq5}, cf.\ Theorem~\ref{p-5}. In order to make sense of \eqref{p-eq5} we have, in particular, to show that $u(\cdot,x)$ and $u(t,\cdot)$ are differentiable and that $A_x u(t,x)$ is well-defined. We start with an auxiliary result showing that \eqref{ihke} gives automatically an estimate for the integrated second derivatives $\int_{\mbb{R}^d} |\partial_{x_i} \partial_{x_i} p_t(x)| \, dx$.

\begin{lemma}\label{p-0}
    Let $(L_t)_{t \geq 0}$ be a $d$-dimensional L\'evy process with transition semigroup $(P_t)_{t \geq 0}$ and density $p_t \in C^2(\mbb{R}^d)$. Let $c: (0,T] \to [0,\infty)$ be a non-negative function.
    \begin{enumerate}
	\item\label{p-0-i}
        $\int_{\mbb{R}^d} |\partial_{x_i} p_t(x)| \, dx \leq c(t)
        \iff \forall \phi \in \mc{B}_b(\mbb{R}^d): \|\partial_{x_i} P_t \phi\| \leq c(t) \| \phi\|_{\infty}$.
        If one \textup{(}hence both\textup{)} of the conditions is satisfied, then $\nabla P_t \phi = P_t (\nabla \phi)$ for any $\phi \in C_b^1(\mbb{R}^d)$.
		
    \item\label{p-0-ii}
        If $\|\partial_{x_i} P_t \phi\|_{\infty} \leq c(t) \|\phi\|_{\infty}$ for all $t \in (0,T]$, $\phi \in \mc{B}_b(\mbb{R}^d)$ and $i \in \{1,\dots,d\}$, then $\|\partial_{x_i} \partial_{x_k} P_{2t} \phi\|_{\infty} \leq c(t)^2 \|\phi\|_{\infty}$ for all $t \in (0,T]$, $\phi \in \mc{B}_b(\mbb{R}^d)$ and $i,k \in \{1,\dots,d\}$.

	\item\label{p-0-iii}
        If $\int_{\mbb{R}^d} |\partial_{x_i} p_t(x)| \, dx \leq c(t)$ for all $t \in (0,T]$ and $i \in \{1,\dots,d\}$, then $\int_{\mbb{R}^d} |\partial_{x_i} \partial_{x_k} p_{2t}(x)| \, dx \leq c(t)^2$ for all $t \in (0,T]$ and $i,k \in \{1,\dots,d\}$.
	\end{enumerate}
\end{lemma}

\begin{proof}
\firstpara{\ref{p-0-i}}
    Suppose that $\int_{\mbb{R}^d} |\partial_{x_i} p_t(x)| \, dx \leq c(t)$ for some $t \in (0,T]$.
    Since
    \begin{gather}\notag
		x \mapsto \int_{\mbb{R}^d} \phi(z) \partial_{x_i} p_t(z-x) \, dz 
	\intertext{is continuous -- this is readily seen with a change of variables $z \mapsto z+x$, dominated convergence, and the fact that $\phi$ is bounded and continuous -- it follows from the differentiation lemma, cf.\ Lemma~\ref{app-1}, that}
    \notag
		x \mapsto P_t \phi(x) = \mbb{E}\phi(x+L_t) = \int_{\mbb{R}^d} \phi(z) p_t(z-x) \, dz
	\intertext{is differentiable and}\label{p-eq6}
		\partial_{x_i} P_t \phi(x) = \int_{\mbb{R}^d} \phi(z) \partial_{x_i} p_t(z-x) \, dz;
	\intertext{thus}\tag{$\star$}\label{p-eq-star}
		\|\partial_{x_i} P_t \phi\|_{\infty} \leq c(t) \|\phi\|_{\infty}, \quad \phi \in \mc{B}_b(\mbb{R}^d).
	\end{gather}
    Suppose that \eqref{p-eq-star} holds. Let $\chi_n \in C_c(\mbb{R}^d)$ be a cut-off function such that $\I_{B(0,n)} \leq \chi_n \leq \I_{B(0,n+1)}$. Applying the differentiation lemma, we find that
    \begin{equation*}
        \partial_{x_i} P_t (\phi \chi_n)(x)
        = \int_{\mbb{R}^d} \phi(z) \chi_n(z) \partial_{x_i} p_t(z-x) \, dz
        =- \int_{\mbb{R}^d} \phi(x+y) \chi_n(x+y) \partial_{y_i} p_t(y) \, dy
	\end{equation*}
	for any $x \in \mbb{R}^d$ and $\phi \in \mc{B}_b(\mbb{R}^d)$. Thus,
    \begin{align*}
		\int_{|y| \leq n} |\partial_{y_i} p_t(y)| \, dy
        &= \sup \left\{ \left| \int_{|y| \leq n} \phi(y) \partial_{y_i} p_t(y) \, dy \right|;\; \phi \in \mc{B}_b(\mbb{R}^d),\; \|\phi\|_{\infty} \leq 1\right\} \\
        &= \sup \left\{ \left| \int_{|y| \leq n} \phi(y) \chi_n(y) \partial_{y_i} p_t(y) \, dy \right|;\; \phi \in \mc{B}_b(\mbb{R}^d),\; \|\phi\|_{\infty} \leq 1 \right\} \\
        &\leq \sup\left\{\|\partial_{x_i} P_t(\phi \chi_n)\|_{\infty};\; \phi \in \mc{B}_b(\mbb{R}^d),\; \|\phi\|_{\infty} \leq 1\right\}
        \leq c(t).
	\end{align*}
    As $n \in \mbb{N}$ is arbitrary, the monotone convergence theorem gives $\int_{\mbb{R}^d} |\partial_{y_i} p_t(y)| \, dy \leq c(t)$. Then \eqref{p-eq6} and the integration by parts formula show that $\nabla P_t \phi = P_t (\nabla \phi)$ for any $\phi \in C_b^1(\mbb{R}^d)$.

\para{\ref{p-0-ii}}
    Fix $\phi \in \mc{B}_b(\mbb{R}^d)$ and $t \in (0,T]$. Since $P_{t} \phi \in C_b^1(\mbb{R}^d)$ it follows from \ref{p-0-i} and the semigroup property that \begin{align*}
		\|\partial_{x_i} \partial_{x_k} P_{2t} \phi\|_{\infty}
		= \|\partial_{x_i} \partial_{x_k} P_{t} P_{t} \phi\|_{\infty}
		= \|\partial_{x_i} P_{t} (\partial_{x_k} P_{t} \phi)\|_{\infty}
		&\leq c(t) \|\partial_{x_k} P_{t} \phi\|_{\infty} \\
		&\leq c(t)^2 \|\phi\|_{\infty}.
	\end{align*}

\para{\ref{p-0-iii}}
    By \ref{p-0-i} and \ref{p-0-ii}, we have $\|\partial_{x_i} \partial_{x_k} P_{2t} \phi\|_{\infty} \leq c(t)^2 \|\phi\|_{\infty}$. Using a very similar reasoning as in the proof of \ref{p-0-i}, we find that
    \begin{equation*}
		\int_{|y| \leq n} |\partial_{y_i} \partial_{y_k} p_{2t}(y)| \, dy \leq c(t)^2;
	\end{equation*}
	applying the monotone convergence theorem completes the proof.
\end{proof}

Recall that we use for a function $g: [0,T] \times \mbb{R}^d \to \mbb{R}$ and $\beta \in (0,1]$ the notation
\begin{equation*}
    \|g\|_{\mc{C}_b^{\beta}(\mbb{R}^d)}
    := \sup_{t \in [0,T]} \sup_{x \in \mbb{R}^d} |g(t,x)| + \sup_{t \in [0,T]} \sup_{x \neq y} \frac{|g(t,x)-g(t,y)|}{|x-y|^{\beta}}.
\end{equation*}

\begin{lemma} \label{p-1}
    Let $(L_t)_{t \geq 0}$ be a L\'evy process as in Theorem~\ref{main1} with generator $(A,\mc{D}(A))$. Let $g \in C([0,T], \mc{C}_b^{\beta}(\mbb{R}^d))$ for $\beta \in (0,1]$ satisfying \eqref{bc}. For every $T>0$ there exists a mapping $u \in C([0,T],\mc{C}_b^{1, \max\{\beta, \,\gamma_0/2\}}(\mbb{R}^d)) \cap C^1([0,T],C_b(\mbb{R}^d))$ solving
    \begin{gather}\label{p-eq7}
		\frac{\partial}{\partial t} u(t,x) = A_x u(t,x) + g(t,x)
        \quad\text{on $[0,T] \times \mbb{R}^d$}
	\intertext{such that $u(0,\cdot)=0$ and}\label{p-eq9}
        \|u\|_{\infty} + \|\nabla u\|_{\mc{C}_b^{\beta}(\mbb{R}^d)} + \|\nabla u\|_{\mc{C}_b^{\gamma_0/2}(\mbb{R}^d)}
        \leq C_T \|g\|_{\mc{C}_b^{\beta}(\mbb{R}^d)}
	\end{gather}
	for some constant $C_T>0$ which does not depend on $g$ and satisfies $\lim_{T\to 0}C_T = 0$.
\end{lemma}

\begin{remark} \label{p-3}
\firstpara{(i)}
    If we define $v(t,x) := u(T-t,x)$ for fixed $T>0$, then $v$ is a solution to the equation with reversed time
    \begin{equation*}
		-\frac{\partial}{\partial t} v(t,x) = A_x v(t,x)+g(t,x), \quad v(T,\cdot)=0.
	\end{equation*}

\para{(ii)}
    In Theorem~\ref{main1} (and hence in Lemma~\ref{p-1}) we assume that the constant $\alpha$ appearing in \eqref{ihke} is strictly larger than $1$ and we require that the balance condition \eqref{bc} is satisfied. Both assumptions are crucial for the proof of Lemma~\ref{p-1}. The assumption $\alpha>1$ is needed to prove that $\nabla u(t,\cdot)$ exists and $\nabla u(t,\cdot) \in \mc{C}^{\beta}_b(\mbb{R}^d)$ whereas \eqref{bc} is used to show that $\nabla u(t,\cdot) \in \mc{C}^{\gamma_0/2}_b(\mbb{R}^d)$. The fact that $\nabla u(t,\cdot) \in \mc{C}^{\beta}_b(\mbb{R}^d)$ will be needed to construct a solution to \eqref{p-eq5} using Picard iterations, see Theorem~\ref{p-5}, and $\nabla u(t,\cdot) \in \mc{C}^{\gamma_0/2}_b(\mbb{R}^d)$ will be used when we apply It\^o's formula, see Proposition~\ref{app-3} and the proof of Theorem~\ref{main1}. Note that $\nabla u(t,\cdot) \in \mc{C}^{\gamma_0/2}_b(\mbb{R}^d) \subseteq \mc{C}^{\gamma_0-1}_b(\mbb{R}^d)$ implies, in particular, that $u(t,\cdot) \in \mc{D}(A)$ for all $t \in [0,T]$, see the remark following \eqref{pre-eq5}.
\end{remark}

\begin{proof}[Proof of Lemma~\ref{p-1}]
	We claim that
    \begin{equation*}
		u(t,x) := \int_{(0,t)} \mbb{E} g(s,x+L_{t-s}) \, ds, \quad t \in [0,T],\; x \in \mbb{R}^d,
	\end{equation*}
	has all the desired properties.

\medskip\noindent
    \textbf{Step 1: $u$ satisfies \eqref{p-eq9}.} For fixed $i \in \{1,\dots,d\}$ and $t \in (0,T]$ define
    \begin{equation*}
        F_t \phi(x)
        := \int_{\mbb{R}^d} \phi(x+y) \partial_{y_i} p_t(y) \, dy
        = \int_{\mbb{R}^d} \phi(y) \partial_{y_i} p_t(y-x) \, dy,
        \quad \phi \in C_b(\mbb{R}^d),\; x \in \mbb{R}^d.
	\end{equation*}
	By Lemma~\ref{p-0},
    \begin{align*}
		\frac{\partial}{\partial x_i} \mbb{E} \phi(x+L_t)
		&= \frac{\partial}{\partial x_i} \int \phi(y) p_t(y-x) \, dy
		= - F_t \phi(x)
	\end{align*}
	and because of \eqref{ihke} we know that
    \begin{equation}\label{p-eq10}
		\|F_t \phi\|_{\infty} \leq \|\phi\|_{\infty} \int_{\mbb{R}^d} |\partial_{y_i} p_t(y)| \, dy \leq c \|\phi\|_{\infty} t^{-1/\alpha}
	\end{equation}
    as well as
    \begin{equation*}
		|F_t \phi(x)-F_t \phi(z)|
		= \left| \int_{\mbb{R}^d} (\phi(x+y)-\phi(z+y)) \partial_{y_i} p_t(y) \, dy \right|
		\leq c \|\phi\|_{\mc{C}_b^{\beta}(\mbb{R}^d)} |x-z|^{\beta} t^{-1/\alpha}
	\end{equation*}
	for some absolute constant $c>0$; therefore,
    \begin{align*}
        \|\partial_{x_i} \mbb{E}\phi(\cdot+L_t)\|_{\mc{C}_b^{\beta}(\mbb{R}^d)}
        = \|F_t \phi\|_{\mc{C}_b^{\beta}(\mbb{R}^d)}
        \leq 2c t^{-1/\alpha} \|\phi\|_{\mc{C}_b^{\beta}(\mbb{R}^d)} \fa \phi \in \mc{C}_b^{\beta}(\mbb{R}^d).
	\end{align*}
    Applying this to $\phi(y) := g(s,y)$ (with $s \in (0,t)$ fixed) it follows from the differentiation lemma and \eqref{ihke} that $\partial_{x_i} u(t,x)$ exists for all $t \in (0,T]$ and
    \begin{gather*}
		\partial_{x_i} u(t,x) = \int_{(0,t)} \partial_{x_i} \mbb{E}g(s,x+L_{t-s}) \, ds
	\intertext{satisfies}
        \|\partial_{x_i} u\|_{\mc{C}_b^{\beta}(\mbb{R}^d)}
        \leq 2c \|g\|_{\mc{C}_b^{\beta}(\mbb{R}^d)} \int_0^T (T-s)^{-1/\alpha} \, ds
        =: C_1 T^{1-1/\alpha} \|g\|_{\mc{C}_b^{\beta}(\mbb{R}^d)};
	\end{gather*}
    (recall that, by assumption, $\alpha>1$). In order to prove $\|\partial_{x_i} u\|_{\mc{C}_b^{\gamma_0/2}(\mbb{R}^d)} \leq C_{2,T} \|g\|_{\mc{C}_b^{\beta}(\mbb{R}^d)}$ we show that
    \begin{equation}\label{p-eq11}
        \|F_t \phi\|_{\mc{C}_b^{\gamma_0/2}(\mbb{R}^d)}
        \leq c' \|\phi\|_{\mc{C}_b^{\beta}(\mbb{R}^d)} t^{-(2+\gamma_0 (1-\beta))/\alpha},
        \quad \phi \in \mc{C}_b^{\beta}(\mbb{R}^d),
	\end{equation}
	for some absolute constant $c'>0$. By the differentiation lemma, we have
    \begin{equation*}
		\partial_{x_k} F_t \phi(x) = \int_{\mbb{R}^d} \phi(y) \partial_{y_k} \partial_{y_i} p_t(y-x) \, dy
	\end{equation*}
	for all $\phi \in C_b(\mbb{R}^d)$, and so, by \eqref{ihke} and Lemma~\ref{p-0}\ref{p-0-iii},
    \begin{equation}\label{p-eq13}
		\|\partial_{x_k} F_t \phi\|_{\infty} \leq c t^{-2/\alpha} \|\phi\|_{\infty} \fa k=1,\dots,d, \; \phi \in C_b(\mbb{R}^d).
	\end{equation}
    If $\gamma_0 \in [1,2)$, then we can use real interpolation to get  $\mc{C}_b^{\gamma_0/2}(\mbb{R}^d) = (C_b(\mbb{R}^d),C_b^1(\mbb{R}^d))_{\gamma_0/2,\infty}$, cf.\ Triebel \cite[Section 2.7.2]{triebel} or Lunardi \cite[Example 1.8]{lunardi}. From \eqref{p-eq10}, \eqref{p-eq13} and the interpolation theorem, see e.g.\ \cite[Section 1.3.3]{triebel} or \cite[Theorem 1.6]{lunardi}, it follows that
    \begin{equation}\label{p-eq15}
		\|F_t \phi\|_{\mc{C}_b^{\gamma_0/2}(\mbb{R}^d)}
		\leq \|F_t \phi\|_{C_b(\mbb{R}^d)}^{1-\gamma_0/2} \|F_t \phi\|_{C_b^1(\mbb{R}^d)}^{\gamma_0/2}
		\leq c \|\phi\|_{\infty} t^{-(1+\gamma_0/2)/\alpha}, \quad \phi \in C_b(\mbb{R}^d).
	\end{equation}
    If $\gamma_0 = 2$, then \eqref{p-eq15} is a direct consequence of \eqref{p-eq13}. On the other hand, another application of the differentiation lemma shows that for any $\phi \in C_b^1(\mbb{R}^d)$
    \begin{gather*}
		\partial_{x_k} F_t \phi(x) = \int_{\mbb{R}^d} \partial_{x_k} \phi(x+y) \partial_{y_i} p_t(y) \, dy
	\intertext{implying}
		\|\partial_{x_k} F_t \phi\|_{\infty} \leq c \|\phi\|_{C_b^1(\mbb{R}^d)} t^{-1/\alpha}, \quad \phi \in C_b^1(\mbb{R}^d).
	\end{gather*}
	Thus,
    \begin{equation} \label{p-eq17}
		\|F_t \phi\|_{\mc{C}_b^{\gamma_0/2}(\mbb{R}^d)}
        \leq c \|\phi\|_{C_b^1(\mbb{R}^d)} t^{-1/\alpha},
        \quad \phi \in C_b^1(\mbb{R}^d).
	\end{equation}
    Using \eqref{p-eq15} and \eqref{p-eq17} we can apply the interpolation theorem once more to find that $F_t$ maps $\mc{C}_b^{\beta}(\mbb{R}^d) = (C_b(\mbb{R}^d),C_b^1(\mbb{R}^d))_{\beta,\infty}$ into $\mc{C}_b^{\gamma_0/2}(\mbb{R}^d)$ and
    \begin{equation*}	
        \|F_t \phi\|_{\mc{C}_b^{\gamma_0/2}(\mbb{R}^d)}
        \leq c \|\phi\|_{\mc{C}_b^{\beta}(\mbb{R}^d)} t^{-(1-\beta)(1+\gamma_0/2)/\alpha - \beta/\alpha}
        = c \|\phi\|_{\mc{C}_b^{\beta}(\mbb{R}^d)} t^{-\kappa}.
	\end{equation*}
    for $\kappa := (2+\gamma_0(1-\beta))/(2\alpha)$; note that $\kappa<1$ because of the balance condition \eqref{bc}. Applying the estimate to $\phi(y) := g(s,y)$, we conclude that
    \begin{equation*}
		\|\partial_{x_i} u\|_{\mc{C}_b^{\gamma_0/2}(\mbb{R}^d)}
		\leq c \|g\|_{\mc{C}_b^{\beta}(\mbb{R}^d)} \int_{0}^T (T-s)^{-\kappa} \, ds
        =: C_2 T^{1-\kappa} \|g\|_{\mc{C}_b^{\beta}(\mbb{R}^d)}.
	\end{equation*}
	
\medskip\noindent
    \textbf{Step 2: $u$ solves \eqref{p-eq7}.}
    By \cite[Theorem 4.1(iii)]{ihke}, we have \begin{equation*} 
    	\mc{C}_{\infty}^{1,\gamma_0-1}(\mbb{R}^d) := \{f \in \mc{C}_b^{1,\gamma_0-1}(\mbb{R}^d); f \in C_{\infty}(\mbb{R}^d), \forall j=1,\ldots,d: \partial_{x_j} f \in C_{\infty}(\mbb{R}^d)\} \subseteq \mc{D}(A);
    \end{equation*}
    recall that $C_{\infty}(\mbb{R}^d)$ is the space of continuous functions vanishing as $|x| \to \infty$. It follows from the proof of Step 1 that \begin{equation*}
        x \mapsto P_{\eps} \phi(x)
        := \mbb{E} \phi(x+L_{\eps}) \in \mc{C}_{\infty}^{1,\gamma_0/2}(\mbb{R}^d)
        \subseteq \mc{C}_{\infty}^{1,\gamma_0-1}(\mbb{R}^d) \subseteq \mc{D}(A)
	\end{equation*}
	for all $\eps>0$ and $\phi \in \mc{C}_b^{\beta}(\mbb{R}^d) \cap C_{\infty}(\mbb{R}^d)$. Since
    \begin{gather*}
		\frac{d}{dt} P_t f = A P_t f \fa f \in \mc{D}(A)
	\intertext{we find}
		\frac{d}{dt} P_{t+\eps} \phi = A P_{t+\eps} \phi
	\intertext{for all $t \geq 0$, $\eps>0$ and $\phi \in \mc{C}_b^{\beta}(\mbb{R}^d) \cap C_{\infty}(\mbb{R}^d)$ which means that}
		\frac{d}{d\tau} P_{\tau} \phi = A P_{\tau} \phi, \quad \tau>0,\; \phi \in \mc{C}_b^{\beta}(\mbb{R}^d) \cap C_{\infty}(\mbb{R}^d).
	\end{gather*}
    Applying this identity to $\phi(x) := g(s,x)$ with $g \in C([0,T],\mc{C}_b^{\beta}(\mbb{R}^d))$ such that $g(s,\cdot) \in C_{\infty}(\mbb{R}^d)$ for all $s \in [0,T]$ shows that $u(t,x) = \int_{(0,t)} \mbb{E}g(s,x+L_{t-s}) \, ds$ is a function in $C^1([0,T],C_b(\mbb{R}^d))$ which solves \eqref{p-eq7}, see \cite[Lemma 7]{proske} for details.

    For an arbitrary $g \in C([0,T],\mc{C}_b^{\beta}(\mbb{R}^d))$ fix a cut-off function $\chi \in C_c^{\infty}(\mbb{R}^d)$ such that $\I_{B(0,1)} \leq \chi \leq \I_{B(0,2)}$ and set $g_n(t,x) := g(t,x) \chi(x/n)$ for $t \in [0,T]$, $x \in \mbb{R}^d$ and $n \in \mbb{N}$. Since $g_n(t,\cdot)$ vanishes at infinity, it follows from the first part that $u_n(t,x) := \int_{(0,t)} \mbb{E}g_n(s,x+L_{t-s}) \, ds$ satisfies
    \begin{gather}\notag
		\frac{\partial}{\partial t} u_n(t,x) = A_x u_n(t,x) + g_n(t,x), \quad u_n(0,x)=0
	\intertext{i.e.}\label{p-eq19}
		u_n(t,x) = \int_{(0,t)} (A_x u_n(s,x)+g_n(s,x)) \, ds.
	\end{gather}
	We are going to show that we can let $n \to \infty$ using the dominated convergence theorem. For any $R>0$ and $t\in [0,T]$ we have
    \begin{equation*}
        |u_n(t,x)-u(t,x)|
        \leq \sup_{s \in [0,T]} \sup_{|y+x| \leq R} |g_n(s,y)-g(s,y)| + 2T \|g\|_{\infty} \mbb{P} \left( \sup_{0 \leq s \leq T} |L_s+x| > R \right)
	\end{equation*}
	and, by Step 1,
    \begin{align*}
		|\partial_{x_i} u_n(t,x)-\partial_{x_i} u(t,x)|
		&\leq \int_0^t\!\! \int_{|y| \leq R} |g_n(s,y)-g(s,y)| \cdot |\partial_{y_i} p_{t-s}(y-x)|  \, dy \, ds \\
		&\quad + 2\|g\|_{\infty} \int_0^t \!\! \int_{|y|>R} |\partial_{y_i} p_{t-s}(y-x)| \, dy \, ds.
	\end{align*}
    Therefore, we can combine the dominated convergence theorem, \eqref{ihke} and the fact that $g_n \to g$ converges uniformly on compact sets to see that \begin{equation}\label{p-eq20}
		\sup_{t \in [0,T]} |u(t,x)-u_n(t,x)| + \sup_{t \in [0,T]} |\nabla_x u(t,x)-\nabla_x u_n(t,x)| \xrightarrow[]{n \to \infty} 0
	\end{equation}
	for all $x \in \mbb{R}^d$. 	Moreover, by Step 1, \begin{align*}
		\|u_n\|_{\mc{C}_b^{1,\gamma_0-1}(\mbb{R}^d)} + \|u\|_{\mc{C}_b^{1,\gamma_0-1}(\mbb{R}^d)}
		&\leq \|u_n\|_{\mc{C}_b^{1,\gamma_0/2}(\mbb{R}^d)} + \|u\|_{\mc{C}_b^{1,\gamma_0/2}(\mbb{R}^d)}\\
		&\leq C \|g\|_{\mc{C}_b^{\beta}(\mbb{R}^d)} + C \|g_n\|_{\mc{C}_b^{\beta}(\mbb{R}^d)}
		\leq 2C \|g\|_{\mc{C}_b^{\beta}(\mbb{R}^d)}.
	\end{align*}
	Observe that we have $\int_{|y| \leq 1} |y|^{\gamma_0} \, \nu(dy)<\infty$ and for any function $f \in \mc{C}_b^{1,\gamma_0-1}(\mbb{R}^d)$
    \begin{align*}
		|f(x+y)-f(x)| \leq 2 \|f\|_{\infty}
        \quad\text{and}\quad
        |f(x+y)-f(x)-\nabla f(x) \cdot y| \leq \|f\|_{\mc{C}_b^{1,\gamma_0-1}(\mbb{R}^d)} |y|^{\gamma_0}.
	\end{align*}
	Therefore, we can use dominated convergence and \eqref{p-eq20} to infer that
    \begin{equation*}
		\sup_{t \in [0,T]} |A_x u_n(t,x)+A_x u(t,x)| \xrightarrow[]{n \to \infty} 0 \fa x \in \mbb{R}^d.
	\end{equation*}
	Letting $n \to \infty$ in \eqref{p-eq19}, we finally get
    \begin{align*}
		u(t,x) = \int_{(0,t)} (A_x u(s,x)+g(s,x)) \, ds, \quad t \in [0,T],\; x \in \mbb{R}^d.
	\end{align*}
	This shows that $u \in C^1([0,T],C_b(\mbb{R}^d))$ solves \eqref{p-eq7}.
\end{proof}

\begin{theorem} \label{p-5}
    Let $(L_t)_{t \geq 0}$ be a $d$-dimensional L\'evy process as in Theorem~\ref{main1} with infinitesimal generator $(A,\mc{D}(A))$, and $b,g \in C([0,T],\mc{C}_b^{\beta}(\mbb{R}^d))$ for $\beta \in (0,1]$ satisfying \eqref{bc}. For sufficiently small $T>0$ there exists a map $u \in C([0,T],\mc{C}_b^{1, \max\{\beta,\, \gamma_0/2\}}(\mbb{R}^d)) \cap C^1([0,T], C_b(\mbb{R}^d))$ solving the equation
    \begin{gather}\label{p-eq21}
    \begin{aligned}
		\frac{\partial}{\partial t} u(t,x) + A_x u(t,x) + b(t,x) \cdot \nabla_x u(t,x) &= -g(t,x) &&\text{on $[0,T) \times \mbb{R}^d$}
        \\ u(T,\cdot) &= 0.
    \end{aligned}
	\intertext{Moreover, $u$ satisfies}\label{p-eq23}
        \|u\|_{\infty} + \|\nabla_x u\|_{\mc{C}_b^{\beta}(\mbb{R}^d)} + \|\nabla_x u\|_{\mc{C}_b^{\gamma_0/2}(\mbb{R}^d)} \leq c(T) \|g\|_{\mc{C}_b^{\beta}(\mbb{R}^d)}
	\end{gather}
	for some constant $c(T)>0$ which does not depend on $b$, $g$ and $c(T) \to 0$ as $T \to 0$.
\end{theorem}

\begin{proof}
    We use Picard iteration to prove the existence of the solution. Choose $T>0$ so small that $2C_T \|b\|_{\mc{C}_b^{\beta}(\mbb{R}^d)} \leq 1/2$ where $C_T$ is the constant appearing in \eqref{p-eq9}, and set $u^{(0)}:=0$. By Lemma~\ref{p-1} and Remark~\ref{p-3}(i) we can define iteratively $u^{(n+1)} \in C([0,T],\mc{C}_b^{\max\{\beta,\, \gamma_0/2\}}(\mbb{R}^d)) \cap C^1([0,T],C_b(\mbb{R}^d))$ such that \eqref{p-eq9} holds and
    \begin{equation*}
        \frac{\partial}{\partial t} u^{(n+1)}(t,x) + A_x u^{(n+1)}(t,x) = - b(t,x) \cdot \nabla_x u^{(n)}(t,x) - g(t,x),
        \quad u^{(n+1)}(T,\cdot)=0.
	\end{equation*}
	Using repeatedly \eqref{p-eq9} we find
    \begin{align*}
		\|u^{(n+1)}-u^{(n)}\|_{\mc{C}_b^{1,\max\{\beta,\, \gamma_0/2\}}(\mbb{R}^d)}
		&\leq C_T \|b \cdot \nabla_x u^{(n)}- b \cdot \nabla_x u^{(n-1)}\|_{\mc{C}_b^{\beta}(\mbb{R}^d)} \\
		&\leq 2 C_T \|b\|_{\mc{C}_b^{\beta}(\mbb{R}^d)} \|\nabla_x u^{(n)}-\nabla_x u^{(n-1)}\|_{\mc{C}_b^{\beta}(\mbb{R}^d)} \\
		&\leq \frac{1}{2}  \|\nabla_x u^{(n)}-\nabla_x u^{(n-1)}\|_{\mc{C}_b^{\beta}(\mbb{R}^d)}
		\leq \dots \leq \frac{1}{2^n} \|g\|_{\mc{C}_b^{\beta}(\mbb{R}^d)},
	\end{align*}
	and, therefore,
    \begin{gather*}
		\sum_{n \geq 1} \|u^{(n+1)}-u^{(n)}\|_{\mc{C}_b^{1,\max\{\beta,\, \gamma_0/2\}}(\mbb{R}^d)}  < \infty.
	\end{gather*}
    Since $C([0,T],\mc{C}_b^{1,\max\{\beta,\, \gamma_0/2\}}(\mbb{R}^d))$ is a Banach space, completeness implies that there is some $u \in C([0,T], \mc{C}_b^{1,\max\{\beta,\, \gamma_0/2\}}(\mbb{R}^d))$ such that $u^{(n)} \to u$ in $C([0,T],\mc{C}_b^{1,\max\{\beta,\, \gamma_0/2\}}(\mbb{R}^d))$. In particular,  by \eqref{pre-eq6},
    \begin{equation*}
		\|A u(t,\cdot) - A u^{(n)}(t,\cdot)\|_{\infty}
		\leq M \|u-u^{(n)}\|_{\mc{C}_b^{1,\gamma_0-1}(\mbb{R}^d)}
		\leq M \|u-u^{(n)}\|_{\mc{C}_b^{1,\gamma_0/2}(\mbb{R}^d)}
		\xrightarrow[]{n \to \infty} 0
	\end{equation*}
	(note that $\gamma_0-1 \leq \gamma_0/2$ as $\gamma_0 \in [1,2]$). Letting $n \to \infty$ in
    \begin{gather*}
		u^{(n)}(t,x) = \int_t^T \left(A_x u^{(n)}(s,x)+b(s,x) \cdot \nabla_x u^{(n)}(s,x) + g(s,x)\right) \, ds
	\intertext{we get}
		u(t,x) = \int_t^T \left(A_x u(s,x) + b(s,x) \cdot \nabla_x u(s,x) + g(s,x)\right) \, ds.
	\end{gather*}
	Using the above estimates, it is not difficult to see that $u$ has all the desired properties.
\end{proof}

We are now ready to prove Theorem~\ref{main1}.
\begin{proof}[Proof of Theorem~\ref{main1}]
    By considering each coordinate of $X_t  \in \mbb{R}^d$ separately, we may assume, without loss of generality, that $d=1$. Fix some sufficiently small $\eps>0$ (we will specify $\eps$ later in the proof), $p \leq \gamma_{\infty}$, $T>0$ and set $T_i := T \frac{i}{L}$, $i=0,\dots,L$.

    If we choose $L=L(\eps) \in \mbb{N}$ sufficiently large, Theorem~\ref{p-5} shows that there exists a function $u_i \in C([T_{i-1},T_i],\mc{C}_b^{1,\max\{\beta,\, \gamma_0/2\}}(\mbb{R})) \cap C^1([T_{i-1},T_i],C_b(\mbb{R}))$ such that
    \begin{align}\label{p-eq27}
    \begin{aligned}
        \frac{\partial}{\partial t} u_i(t,x) + A_x u_i(t,x) + b(t,x) \frac{\partial}{\partial x} u_i(t,x) &= - b(t,x)
        &&\text{on $[T_{i-1},T_i) \times \mbb{R}$} \\
		u_i(T_i,\cdot) &= 0
    \end{aligned}
	\intertext{and}\label{p-eq29}
		\|u_i\|_{\infty} + \|\partial_x u_i\|_{\mc{C}_b^{\beta}(\mbb{R}^d)} + \|\partial_x u_i\|_{\mc{C}_b^{\gamma_0/2}(\mbb{R}^d)} \leq \eps.
	\end{align}
	Denote by $(X^{(n)}_t)_{t \geq 0}$ the Euler--Maruyama approximation, i.e.\
    \begin{equation*}
		X_t^{(n)} = x + \int_0^t b(\eta_n(s),X_{\eta_n(s)-}^{(n)}) \, ds + L_t
	\end{equation*}
	where $\eta_n(s) := T \frac{i}{n}$ for $s \in [T \frac{i}{n}, T \frac{i+1}{n})$. We are going to show that
    \begin{equation*}
		\mbb{E} \left( \sup_{0 \leq t \leq T} |X_t^{(n)}-X_t^{(m)}|^p \right) \xrightarrow[]{m,n \to \infty} 0.
	\end{equation*}
	Applying It\^o's formula, cf.\ Proposition~\ref{app-3}, \marginpar{\footnotesize Please see the \newline correction on \\ the final page} it follows from \eqref{p-eq27} that
	\begin{align} \label{p-eq30} \begin{aligned}
		\int_{T_{i-1}}^t b(s,X_s^{(n)}) \, ds
		&= u_i(T_{i-1},X_{T_{i-1}}^{(n)}) - u_i(t,X_t^{(n)}) \\
		&\quad\mbox{} - \int_{T_{i-1}}^t \!\! \int_{|y| < 1} \big(u_i(s,X_{s-}^{(n)}+y)-u_i(s,X_{s-}^{(n)})\big) \, \tilde{N}(dy,ds) \\
		&\quad\mbox{} - \int_{T_{i-1}}^t \!\! \int_{|y| \geq 1} \big(u_i(s,X_{s-}^{(n)}+y)-u_i(s,X_{s-}^{(n)})\big) \,  N(dy,ds) \\
		&\quad\mbox{} - \int_{T_{i-1}}^t \big( b(\eta_n(s),X_{\eta_n(s)}^{(n)}) -b(s,X_s^{(n)}) \big) \partial_x u_i(s,X_s^{(n)}) \, ds
    \end{aligned}\end{align}
    for any $t \in [T_{i-1},T_i)$, $i=0,\dots,L$, where $\tilde{N}(dy,ds) = N(dy,ds)-\nu(dy) \, ds$ denotes the compensated jump measure of the L\'evy process $(L_t)_{t \geq 0}$. Fix $i \in \{0,\dots,L\}$, $t \in [T_{i-1},T_i]$ and $m,n \in \mbb{N}$. Observing that
    \begin{align*}
        |X_t^{(m)}-X_t^{(n)}|
        \leq |X_{T_{i-1}}^{(m)}-X_{T_{i-1}}^{(n)}| + \left| \int_{T_{i-1}}^t b(\eta_m(s),X_{\eta_m(s)}^{(m)}) \, ds
                - \int_{T_{i-1}}^t b(\eta_n(s),X_{\eta_n(s)}^{(n)}) \, ds \right|,
    \end{align*}
    we get \marginpar{\footnotesize Please see the \newline correction on \\ the final page}
   \begin{align*}
		|X_t^{(m)}-X_t^{(n)}|^p \leq C |X_{T_{i-1}}^{(m)}-X_{T_{i-1}}^{(n)}|^p + C (I_1+I_2+I_3+I_{4,1}+I_{4,2}+I_{5})
	\end{align*}
	for some constant $C=C(p)>0$ and the following (integral) expressions
    \begin{align*}
		I_1 &:= \left|u_i(T_{i-1},X_{T_{i-1}}^{(m)})-u_i(T_{i-1},X_{T_{i-1}}^{(n)})\right|^p + \left|u_i(t,X_t^{(m)})-u_i(t,X_t^{(n)})\right|^p \\
		I_2 &:= \left| \int_{T_{i-1}}^t \!\! \int_{|y|<1} H_i(s,y) \, \tilde{N}(dy,ds) \right|^p \\
		I_3 &:= \left| \int_{T_{i-1}}^t \!\! \int_{|y| \geq 1} H_i(s,y) \, N(dy,ds) \right|^p \\
		I_{4,1} &:= \left| \int_{T_{i-1}}^t \big( b(\eta_n(s),X_{\eta_n(s)}^{(n)}) -b(s,X_s^{(n)}) \big) \partial_x u_i(s,X_s^{(n)}) \, ds \right|^p \\
        I_{4,2} &:=  \left| \int_{T_{i-1}}^t \big( b(\eta_m(s),X_{\eta_m(s)}^{(m)}) -b(s,X_s^{(m)}) \big) \partial_x u_i(s,X_s^{(m)}) \, ds \right|^p \\
        I_{5} &:= \left| \int_{T_{i-1}}^t (b(\eta_n(s),X_{\eta_n(s)}^{(n)})-b(s,X_s^{(n)}) \, ds \right|^p + \left| \int_{T_{i-1}}^t (b(\eta_m(s),X_{\eta_m(s)}^{(m)})-b(s,X_s^{(m)}) \, ds \right|^p
	\end{align*}
	and \begin{equation*}
		H_i(s,y) := \big( u_i(s,X_{s-}^{(m)}+y) - u_i(s,X_{s-}^{(m)}) \big) - \big(u_i(s,X_{s-}^{(n)}+y) -u_i(s,X_{s-}^{(n)}) \big).
	\end{equation*}
    We estimate the terms separately. Because of \eqref{p-eq29}, we have $\|\partial_x u_i\|_{\infty} \leq \eps$, and therefore an application of the mean value theorem shows
    \begin{equation*}
		I_1 \leq \eps^p \left|X_{T_{i-1}}^{(m)}-X_{T_{i-1}}^{(n)}\right|^p + \eps^p \left|X_t^{(m)}-X_t^{(n)}\right|^p.
	\end{equation*}
    Moreover, it follows from \eqref{p-eq29} and the fact that $b(t,x)$ is $\beta$-H\"{o}lder-continuous with respect to $x$ and $\eta$-H\"{o}lder continuous with respect to $t$ that
    \begin{align*}
		I_{4,1}
        &\leq C \left| \int_{T_{i-1}}^t \big(b(\eta_n(s),X_{\eta_n(s)}^{(n)})-b(s,X_{\eta_n(s)}^{(n)})\big) \partial_x u_i(s,X_s^{(n)}) \, ds \right|^p \\
		&\quad + C \left| \int_{T_{i-1}}^t \big(b(s,X_{\eta_n(s)}^{(n)})-b(s,X_s^{(n)})\big) \partial_x u_i(s,X_s^{(n)}) \, ds \right|^p \\
		&\leq C_4 n^{-p \eta} \eps^p + C_4' \eps^p \sup_{T_{i-1} \leq t \leq T_i} |X_{\eta_n(t)}^{(n)}-X_t^{(n)}|^{\beta p}.
	\end{align*}
	The same estimate holds for $I_{4,2}$ with $n$ replaced by $m$. In exactly the same fashion we get
    \begin{equation*}
		I_5
        \leq C_5 n^{-\eta p} + C_5' \eps^p \sup_{T_{i-1} \leq t \leq T_i} |X_{\eta_n(t)}^{(n)}-X_{t}^{(n)}|^{\beta p}
            + C_5' \eps^p \sup_{T_{i-1} \leq t \leq T_i} |X_{\eta_m(t)}^{(m)}-X_{t}^{(m)}|^{\beta p}.
	\end{equation*}
	In order to estimate $I_2$ and $I_3$, we use Taylor's formula and \eqref{p-eq29}
    \begin{equation}\label{p-eq31}\begin{aligned}
		|H_i(s,y)|
        &\leq |X_{s-}^{(m)}-X_{s-}^{(n)}| \int_0^1 \left| \frac{\partial}{\partial x} u_i\big(s,X_s^{(m)}+y+\tau (X_{s-}^{(n)}-X_{s-}^{(m)})\big)\right.\\
        &\qquad\qquad\qquad\qquad\qquad\qquad\mbox{} - \left.\frac{\partial}{\partial x} u_i\big(s,X_{s-}^{(m)}+\tau(X_{s-}^{(n)}-X_{s-}^{(m)})\big) \right| \, d\tau \\
		&\leq \eps \min \big\{|X_{s-}^{(m)}-X_{s-}^{(n)}| \, |y|^{\gamma_0/2}, \: 2 |X_{s-}^{(m)}-X_{s-}^{(n)}|\big\}.
	\end{aligned}\end{equation}
	Applying the Burkholder--Davis--Gundy inequality, cf.\ Theorem~\ref{pre-1}, we find
    \begin{align*}
		&\mbb{E} \left( \sup_{T_{i-1} \leq t \leq T_i} |I_2| \right) \\
        &\leq C_2 \mbb{E} \left[ \left( \int_{T_{i-1}}^{T_i} \!\! \int_{|y|<1} |H_i(s,y)|^2 \, \nu(dy) \, ds \right)^{p/2} \right]
            + C_2 \mbb{E} \left( \int_{T_{i-1}}^{T_i} \!\! \int_{|y|<1} |H_i(s,y)|^p \, \nu(dy) \, ds \right) \I_{[2,\infty)}(p) \\
        &\leq C_2' \eps^p
            \left( \left[ \int_{|y|<1} |y|^{\gamma_0}\,\nu(dy) \right]^{p/2} + \int_{|y| <1} |y|^{\gamma_0}\,\nu(dy)\right)
            \mbb{E} \left( \sup_{T_{i-1} \leq t \leq T_i} |X_s^{(m)}-X_s^{(n)}|^p \right)
	\end{align*}
    for some absolute constants $C_2, C_2'>0$. In order to estimate $I_3$ we distinguish between two cases. If $p \in (0,1)$, then $(x+y)^p \leq x^p + y^p$ for all $x,y \geq 0$, and therefore by \eqref{p-eq31}
    \begin{align*}
		\mbb{E} \left( \sup_{T_{i-1} \leq t \leq T_i} |I_3| \right)
        &= \mbb{E} \Bigg( \sup_{T_{i-1} \leq t \leq T_i} \Bigg| \sum_{\substack{s \in [T_{i-1},t] \\ |\Delta L_s| \geq 1}} H_i(s,\Delta L_s) \Bigg|^p \Bigg)\\
        &\leq \mbb{E} \Bigg( \sum_{\substack{s \in [T_{i-1},T_i] \\ |\Delta L_s| \geq 1}} |H_i(s,\Delta L_s)|^p \Bigg) \\
		&= \mbb{E} \left( \int_{T_{i-1}}^{T_i} \!\! \int_{|y| \geq 1} |H_i(s,y)|^p \, N(dy,ds) \right) \\
		&= \mbb{E} \left( \int_{T_{i-1}}^{T_i}\!\! \int_{|y| \geq 1} |H_i(s,y)|^p \, \nu(dy) \, ds \right) \\
		&\leq 2^p \eps^p \nu(B(0,1)^c) \mbb{E} \left( \sup_{T_{i-1} \leq t \leq T_i} |X_t^{(m)}-X_t^{(n)}|^p \right).
	\end{align*}
	If $p \geq 1$, then $\int_{|y| \geq 1} |y| \, \nu(dy)<\infty$, and so
    \begin{equation*}
		I_3
        \leq C \left| \int_{T_{i-1}}^{t} \!\! \int_{|y| \geq 1} H_{i}(s,y) \, \tilde{N}(dy,ds) \right|^p + C \left| \int_{T_{i-1}}^t \!\! \int_{|y| \geq 1} H_i(s,y) \, \nu(dy) \, ds \right|^p.
	\end{equation*}
	By the Burkholder--Davis--Gundy inequality and \eqref{p-eq31}, there exist absolute constants $C_3,C_3'>0$ such that
    \begin{align*}
    	&\mbb{E} \left( \sup_{T_{i-1} \leq t \leq T_i} I_3 \right)\\
        &\quad\leq C_3 \mbb{E} \left( \left[ \int_{T_{i-1}}^{T_i} \!\! \int_{|y| \geq 1} |H_i(s,y)|^2 \, \nu(dy) \, ds \right]^{p/2} \right)
            + C_3 \mbb{E} \left( \int_{T_{i-1}}^{T_i} \!\! \int_{|y| \geq 1} |H_i(s,y)|^p \, \nu(dy) \, ds \right) \\
    	&\quad\leq C_3' \eps^p \big(\nu(B(0,1)^c)^{p/2} + \nu(B(0,1)^c)\big) \mbb{E} \left( \sup_{T_{i-1} \leq t \leq T_i} |X_t^{(m)}-X_t^{(n)}|^p \right).
	\end{align*}	
	Combining the above estimates we conclude that there exist constants $c_1, c_2>0$ (not depending on $\eps$, $m$, $n$, $L$, $i$) such that
    \begin{align*}
        &\mbb{E} \left( \sup_{T_{i-1} \leq t \leq T_i}| X_t^{(m)}-X_t^{(n)}|^p \right)\\
        &\quad\leq \eps^p c_1 \mbb{E} \left( \sup_{T_{i-1} \leq t \leq T_i}| X_t^{(m)}-X_t^{(n)}|^p \right)
            + c_2 \mbb{E} \left( |X_{T_{i-1}}^{(m)}-X_{T_{i-1}}^{(n)}|^p \right)
            + \frac{c_2}{n^{p \eta}} + \frac{c_2}{m^{p \eta}} \\
		&\qquad + c_2 \mbb{E} \left( \sup_{T_{i-1} \leq t \leq T_i} |X_{\eta_n(t)}^{(n)}-X_t^{(n)}|^{p \beta} \right)
            + c_2 \mbb{E} \left( \sup_{T_{i-1} \leq t \leq T_i} |X_{\eta_m(t)}^{(m)}-X_t^{(m)}|^{p \beta} \right).
	\end{align*}
	Thus,
    \begin{align*}
		&(1-\eps^p c_1) \mbb{E} \left( \sup_{T_{i-1} \leq t \leq T_i} | X_t^{(m)}-X_t^{(n)}|^p \right) \\
		&\quad\leq c_2 \mbb{E} \left( |X_{T_{i-1}}^{(m)}-X_{T_{i-1}}^{(n)}|^p \right)
            + \frac{2c_2}{N^{p \eta}}
            + 2c_2 \sup_{n \geq N} \mbb{E} \left( \sup_{T_{i-1} \leq t \leq T_i} |X_{\eta_n(t)}^{(n)}-X_t^{(n)}|^{p \beta} \right)
	\end{align*}
    for any $m,n \geq N$. Choose $\eps>0$ so small that $1- \eps^p c_1 \geq 1/2$. By the very definition of the Euler--Maruyama approximation, we have \begin{equation*}
		X_{\eta_n(t)}^{(n)}-X_t^{(n)} = \int_{\eta_n(t)}^t b(\eta_n(s),X_{\eta_n(s)-}^{(n)}) \, ds + L_{\eta_n(t)}-L_t.
	\end{equation*}
    Using $L_{t}-L_{\eta_n(t)} \stackrel{d}{=} L_{t-\eta_n(t)}$ and fractional moment estimates for L\'evy processes, see \cite[Section 5]{moments}, we can find a constant $c_3>0$ such that
    \begin{align*}
        \mbb{E} \left( \sup_{T_{i-1} \leq t \leq T_i} |X_{\eta_n(t)}^{(n)}-X_t^{(n)}|^{p \beta} \right)
        &\leq (2\|b\|_{\infty})^{p\beta} n^{-p \beta} + 2^{p \beta} \mbb{E} \left( \sup_{T_{i-1} \leq t \leq T_i} |L_{t-\eta_n(t)}|^{p \beta} \right) \\
        &\leq (2\|b\|_{\infty})^{p\beta} n^{-p \beta} + 2^{p \beta} \mbb{E} \left( \sup_{s \leq 1/n} |L_s|^{p \beta} \right)\\
        &\leq c_3 n^{-\min\{1,\,p \beta/\gamma_0\}}.
	\end{align*}
	Hence,
    \begin{align*}
		\mbb{E} \left( \sup_{T_{i-1} \leq t \leq T_i} | X_t^{(m)}-X_t^{(n)}|^p \right)
		\leq 2c_2 \mbb{E} \left( |X_{T_{i-1}}^{(m)}-X_{T_{i-1}}^{(n)}|^p \right)  + 8c_2 c_3 N^{-\min\{1,p \beta/\gamma_0,p \eta\}}.
	\end{align*}
	Using this estimate iteratively for $i=1,\dots,L$, we conclude that there exists a constant $c_4=c_4(L)>0$ such that
    \begin{equation*}
		\mbb{E} \left( \sup_{T_{i-1} \leq t \leq T_i} |X_t^{(m)}-X_t^{(n)}|^p \right)
        \leq c_4 N^{-\min\{1,p \beta/\gamma_0, p \eta\}}
	\end{equation*}
	for all $m,n \geq N$. Thus,
    \begin{align}\label{p-eq33}
        \mbb{E} \left( \sup_{0 \leq t \leq T} |X_t^{(m)}-X_t^{(n)}|^p \right)
        \leq \sum_{i=1}^L \mbb{E} \left( \sup_{T_{i-1} \leq t \leq T_i} |X_t^{(m)}-X_t^{(n)}|^p \right)
        \leq c_4 L N^{-\min\{1,p \beta/\gamma_0, p \eta\}}
	\end{align}
	for all $m,n \geq N$; this means, in particular, that
    \begin{equation*}
		\mbb{E} \left( \sup_{0 \leq t \leq T} |X_t^{(m)}-X_t^{(n)}|^p \right) \xrightarrow[]{m,n \to \infty} 0.
	\end{equation*}
    This implies that there exists a stochastic process $(X_t)_{t \in [0,T]}$ and a subsequence $(n_k)_{k \in \mbb{N}}$ such that $\sup_{t \in [0,T]} |X_t^{(n_k)}-X_t| \to 0$ almost surely as $k \to \infty$, cf.\ Lemma~\ref{app-5}.  Letting $k \to \infty$ in
    \begin{gather*}
		X_t^{(n_k)} -x = \int_0^t b(\eta_{n_k}(s),X_{\eta_{n_k}(s)}^{(n_k)}) \, ds + L_t
	\intertext{we find}
		X_t -x = \int_0^t b(s,X_{s-}) \,ds + L_t.
	\end{gather*}
	Moreover, it follows from Fatou's lemma and \eqref{p-eq33} that
    \begin{equation*}
		\mbb{E} \left( \sup_{0 \leq t \leq T} |X_t-X_t^{(n)}|^p \right)
        \leq c_4' n^{-\min\{1,\,p \beta/\gamma_0,\, p\eta\}}.
	\end{equation*}
    This proves the existence of a solution to \eqref{sde} satisfying \eqref{ineq}.

    In order to show uniqueness, we assume that $(Y_t)_{t \geq 0}$ is a further solution to \eqref{sde}. Applying It\^o's formula to $u_i(t,Y_t)$ with $u_i$ as in \eqref{p-eq27}, we get a similar expression as in \eqref{p-eq30} for $\int_0^t b(s,Y_s) \, ds$, and a very similar reasoning as in the first part of the proof to shows that
    \begin{equation}
		\mbb{E} \left( \sup_{0 \leq t \leq T} |Y_t-X_t^{(n)}|^p \right) \leq c_5 {n^{-\min\{1,\,p \beta/\gamma_0, \,p \eta\}}} \label{p-eq37}
	\end{equation}
	for all $n \in \mbb{N}$. Thus, by Fatou's lemma,
    \begin{gather*}
		\mbb{E} \left( \sup_{0 \leq t \leq T} |Y_t-X_t|^p \right) = 0.
    \qedhere
	\end{gather*}
\end{proof}

\begin{proof}[Proof of Corollary~\ref{main3}]
	From \cite[Theorem 1.3]{ssw} it follows that the semigroup $P_t \phi(x) := \mbb{E}\phi(x+L_t)$ satisfies
    \begin{equation*}
		\|\nabla P_t \phi\|_{\infty} \leq c t^{-1/\alpha} \|\phi\|_{\infty} \fa \phi \in \mc{B}_b(\mbb{R}^d).
	\end{equation*}
    By Lemma~\ref{p-0}, this implies $\int_{\mbb{R}^d} |\partial_i p_t(x)| \, dx \leq ct^{-1/\alpha}$ for all $i \in \{1,\dots,d\}$. Applying Theorem~\ref{main1} finishes the proof.
\end{proof}

The remaining part of this section is devoted to the proof of Corollary~\ref{main2}. From now on $(S_t)_{t \geq 0}$ denotes a subordinator with Laplace exponent (Bernstein function) $f$, $(B_t^{(d)})_{t \geq 0}$ is a $d$-dimensional Brownian motion and $L_t^{(d)} := B_{S_t}^{(d)}$ is the process subordinate to Brownian motion. Note that $(L_t^{(d)})_{t \geq 0}$ is a $d$-dimensional L\'evy process with characteristic exponent $\psi(\xi) = f(|\xi|^2)$, $\xi \in \mbb{R}^d$. If $f$ satisfies the Hartman--Wintner condition
\begin{equation}\label{hw}
	\lim_{r \to \infty} \frac{f(r)}{\log(1+r)} = \infty,
\end{equation}
then $L_t^{(d)}$  has for all $t>0$ a transition density $p_t^{(d)}$ see e.g.\ \cite{knop}, and $p_t^{(d)}$ is isotropic, i.e.\ $p_t^{(d)}(x)$ depends only on $|x|$; in abuse of notation we write $p_t^{(d)}(x) = p_t^{(d)}(|x|)$. Using polar coordinates one finds, cf.\ Matheron \cite[pp.\ 33--4]{matheron} and \cite{schoenberg},
\begin{equation}\label{diff}
	\frac{d}{dr} p_t^{(d)}(r)
    = -2 \pi r p_t^{(d+2)}(r), \quad r>0,\; d \geq 1.
\end{equation}

\begin{lemma} \label{p-7}
    Let $f$ be a Bernstein function satisfying the Hartman--Wintner condition \eqref{hw}, and let $(L_t^{(d)})_{t \geq 0}$ be a $d$-dimensional L\'evy process with characteristic exponent $\psi(\xi) = f(|\xi|^2)$, $\xi \in \mbb{R}^d$, for $d \geq 1$. If there exist constants $c>0$, $\alpha>0$ such that \begin{equation} \label{p-eq51}
		\mbb{E} \left( \big|L_t^{(d+2)}\big|^{-1} \right) \leq ct^{-1/\alpha} \fa t \in (0,T],
	\end{equation}
	then the transition density of $L_t^{(d)}$, $t>0$, satisfies \eqref{ihke}.
\end{lemma}	
\begin{proof}
Denote by $p_t^{(d)}(x) = p_t^{(d)}(|x|)$, $x \in \mbb{R}^d$, the transition density of $(L_t^{(d)})_{t \geq 0}$. Using polar coordinates and \eqref{diff}, we find for each $i=1,\dots, d$,
\begin{align*}
    \int_{\mbb{R}^d} |\partial_{x_i} p_t^{(d)}(x)| \, dx
	\leq 2\pi \int_{\mbb{R}^d} |x| p_t^{(d+2)}(|x|) \, dx
	&= 2\pi \sigma_d \int_{(0,\infty)} r p_t^{(d+2)}(r) r^{d-1} \, dr \\
	&= 2 \pi \sigma_d \int_{(0,\infty)} \frac{1}{r} p_t^{(d+2)}(r) r^{(d+2)-1} \, dr \\
	&= \frac{\sigma_d}{\sigma_{d+2}} \int_{\mbb{R}^{d+2}} \frac{1}{|x|} p_t^{(d+2)}(|x|) \, dx \\
	&= \frac{\sigma_d}{\sigma_{d+2}} \mbb{E} \left( \big|L_t^{(d+2)}\big|^{-1} \right)
\end{align*}
where $\sigma_d$ is the surface volume of the unit sphere $\mbb{S}^{d-1}\subseteq\mbb{R}^d$. Hence, by \eqref{p-eq51},
\begin{gather*}
    \int_{\mbb{R}^d} |\partial_{x_i} p_t(x)| \, dx \leq c' t^{-1/\alpha} \fa t \in (0,T],\; i=1,\dots,d.
\qedhere
\end{gather*}
\end{proof}

\begin{remark}\label{p-9}
	More generally, the condition
    \begin{gather*}
	   \mbb{E} \left(\big|L_t^{(d+2i)}\big|^{-i}\right) \leq ct^{-i/\alpha},
        \quad i=1,\dots,k,\; t \in (0,T]
	\intertext{guarantees that}
	   \int_{\mbb{R}^d} |\partial^{\gamma}_x p_t(x)| \, dx
        \leq c' t^{-|\gamma|/\alpha}
        \fa \gamma \in \mbb{N}_0^d,\; |\gamma| \leq k,\; t \in (0,T].
	\end{gather*}
\end{remark}

\begin{proof}[Proof of Corollary~\ref{main2}]
	By assumption, there exists some $c>0$ such that $\psi(\xi) \geq c |\xi|^{2\rho}$ for large $|\xi|$, and therefore
    \begin{equation*}
		p_t(x) = (2\pi)^{-d} \int_{\mbb{R}^d} e^{-ix \cdot \xi} e^{-t \psi(\xi)} \, d\xi, \quad x \in \mbb{R}^d
	\end{equation*}
    is the density of $L_t$; by the differentiation lemma, $p_t$ is twice continuously differentiable. To prove that the density of $L_t = L_t^{(d)} = B_{S_t}^{(d)}$, $t>0$, satisfies \eqref{ihke} for $\alpha :=2 \rho$, it suffices by Lemma~\ref{p-7} to show that \eqref{p-eq51} holds for $\alpha = 2 \rho$. To this end, we recall that for any $\kappa>0$ there exists a constant $C>0$ such that
    \begin{equation}\label{p-eq53}
		\mbb{E}(S_t^{-\kappa}) \leq C \min\{t,\,1\}^{-\kappa/\rho}, \quad t \geq 0,
	\end{equation}
	cf.\ \cite[Theorem 3.17]{deng}. As $(B_t^{(d)})_{t \geq 0}$ and $(S_t)_{t \geq 0}$ are independent, we get by the scaling property of Brownian motion
    \begin{align*}
		\mbb{E}\left( \big|L_t^{(d+2)}\big|^{-1} \right)
		= \mbb{E}\left( \big|B_{S_t}^{(d+2)}\big|^{-1} \right)
		= \mbb{E}\left(\big|\sqrt{S_t} B_1^{(d+2)}\big|^{-1} \right)
		= \mbb{E}\left(S_t^{-1/2}\right) \mbb{E}\left(|B_1^{(d+2)}|^{-1}\right).
	\end{align*}
	Note that
    \begin{equation*}
		\mbb{E}\left(\big|B_1^{(d+2)}\big|^{-1}\right)
        = \int_{\mbb{R}^{d+2}} \frac{1}{|z|} \frac{1}{(2\pi)^{(d+2)/2}} \exp \left(- \frac{|z|^2}{2} \right) \, dz
        < \infty
	\end{equation*}
	as $1< d+2$. Because of \eqref{p-eq53} we get \eqref{p-eq51}, hence \eqref{ihke}, for $\alpha = 2 \rho$.

	Finally, since
    \begin{gather*}
		\int_{(0,1)} r^{\delta_0} \, \mu(dr) + \int_{(1,\infty)} r^{\delta_{\infty}} \, \mu(dr)< \infty
	\intertext{implies}
		\int_{B(0,1)} |y|^{2\delta_0} \, \nu(dy) + \int_{B(0,1)^c} |y|^{2\delta_{\infty}} \, \nu(dy)<\infty,
	\end{gather*}
    the assumptions of Theorem~\ref{main1} are satisfied for $\alpha := 2\rho$, $\gamma_0 := 2 \delta_0$, $\gamma_{\infty} := 2 \delta_{\infty}$, and this completes the proof.
\end{proof}

\section{Examples} \label{ex}

The following lemma is useful if one wants to verify the assumptions of Theorem~\ref{main1} and Corollary~\ref{main2}. It shows how the growth of the characteristic exponent at $0$ (resp., at infinity) is related to the existence of moments of the L\'evy measure at infinity (resp., at $0$).

\begin{lemma} \label{ex-1}
    Let $\psi: \mbb{R}^d \to \mbb{C}$ be a continuous negative definite function with L\'evy triplet $(\ell,0,\nu)$, and let $f$ be a Bernstein function with characteristics $(0,\mu)$.
    \begin{enumerate}
	\item\label{ex-1-i}
        If $\mu(dy) \geq c |y|^{-1-\rho} \, dy$ on $B(0,1)$ for some $c>0$ and $\rho \in (0,1)$, then
        \begin{equation*}
			\liminf_{\lambda \to \infty} \frac{f(\lambda)}{\lambda^{\rho}}>0.
		\end{equation*}

	\item\label{ex-1-ii}
    \begin{enumerate}
		\item\label{ex-1-ii-a}
            $f(r) \leq c r^{\delta}$ for all $r \geq 1$ implies $\int_{(0,1)} r^{\delta+\eps} \, \mu(dr)<\infty$ for any $\eps>0$.
		\item\label{ex-1-ii-b}
            $f(r) \leq c r^{\delta}$ for all $r \in [0,1]$ implies $\int_{(1,\infty)} r^{\delta-\eps} \, \mu(dr)<\infty$ for any $\eps>0$.
	\end{enumerate}
	
    \item\label{ex-1-iii}
    \begin{enumerate}
		\item\label{ex-1-iii-a}
            $|\re \psi(\xi)| \leq c |\xi|^{\alpha}$ for all $|\xi| \geq 1$ implies $\int_{B(0,1)} |y|^{\alpha+\eps} \, \nu(dy)<\infty$ for any $\eps>0$.
		\item\label{ex-1-iii-b}
            $|\re \psi(\xi)| \leq c |\xi|^{\alpha}$ for all $|\xi| \leq 1 $ implies $\int_{B(0,1)^c} |y|^{\alpha-\eps} \, \nu(dy)<\infty$ for any $\eps>0$.
	\end{enumerate}
    \end{enumerate}
\end{lemma}

\begin{proof}
\firstpara{\ref{ex-1-i}}
Fix $\lambda>1$. As $1-e^{-\lambda r} \geq 0$ for $r \geq 0$, we have  \begin{align*}
		f(\lambda) = \int_{(0,\infty)} (1-e^{-\lambda r}) \, \mu(dr)
		&\geq \int_{(0,\lambda^{-1})} (1-e^{-\lambda r}) \, \mu(dr) \\
		&\geq c \int_{(0,\lambda^{-1})} (1-e^{-\lambda r}) \,\frac{dr}{r^{1+\rho}}.
	\end{align*}
	Changing variables according to $s := \lambda r$, we find that the right-hand side equals $c' \lambda^{\rho}$ for some strictly positive constant $c'$.

\para{\ref{ex-1-ii-a}}
    If $\delta+\eps \geq 1$ there is nothing to show since $\int_{(0,1)} r \, \mu(dr)<\infty$. For $\delta+\eps \in (0,1)$ we use the formula
    \begin{equation}\label{ex-eq7}
	   r^{\delta+\eps} = \frac{\delta+\eps}{\Gamma(1-\delta-\eps)} \int_{(0,\infty)} (1-e^{-rs}) \,\frac{ds}{s^{1+\delta+\eps}}
	\end{equation}
	It is not difficult to see that this implies
    \begin{equation*}
		r^{\delta+\eps} \leq C \int_{(1,\infty)} (1-e^{-rs}) \,\frac{ds}{s^{1+\delta+\eps}}  \fa r \in (0,1)
	\end{equation*}
	for some constant $C>(\delta+\eps)/\Gamma(1-\delta-\eps)$. Applying Tonelli's theorem we find
    \begin{gather*}
		\int_{(0,1)} r^{\delta+\eps} \, \mu(dr)
		\leq C\int_{(1,\infty)} \!\! \int_{(0,1)} (1-e^{-rs}) \, \mu(dr) \, \frac{ds}{s^{1+\delta+\eps}}
		\leq C \int_{(1,\infty)} \frac{f(s)}{s^{1+\delta+\eps}} \, ds < \infty.
	\end{gather*}

\para{\ref{ex-1-ii-b}}
    Since $f$ grows at most linearly and $\int_{|y| \geq 1} \, \mu(dy)<\infty$, we can assume without loss of generality that $\delta \in (0,1]$ and $\delta-\eps>0$. It follows from \eqref{ex-eq7} (with $\eps$ replaced by $-\eps$)  that there exists a constant $c'>0$ such that
    \begin{equation*}
		r^{\delta-\eps} \leq c' \int_{(0,1)} (1-e^{-rs})\, \frac{ds}{s^{1+\delta-\eps}} \fa r \geq 1.
	\end{equation*}
	Applying Tonelli's theorem once again shows
    \begin{align*}
		\int_{(1,\infty)} r^{\delta-\eps} \, \mu(dr)
		\leq c' \int_{(0,1)} \frac{f(s)}{s^{1+\delta-\eps}} \, ds < \infty.
	\end{align*}
	
\para{\ref{ex-1-iii}}
		The reasoning is very similar to the proof of (ii); use that for $\alpha\in (0,2)$
        \begin{equation*}
			|\xi|^{\alpha}
            =
            \frac{\alpha 2^{\alpha-1}\Gamma\left(\frac{\alpha+d}{2}\right)}{\pi^{d/2}\,\Gamma\left(1-\frac\alpha2\right)} \int_{\mbb{R}^d \setminus \{0\}} (1-\cos(\xi y)) \frac{dy}{|y|^{d+\alpha}}, \quad \xi \in \mbb{R}^d;
		\end{equation*}
		see also \cite[Lemma A.1]{ihke}.
\end{proof}

Combining Lemma~\ref{ex-1} with Corollary~\ref{main2} we get the following statement.

\begin{example} \label{ex-3}
    Let $(L_t)_{t \geq 0}$ be a $d$-dimensional L\'evy process with one of the following characteristic exponents $\psi:\mbb{R}^d\to\mbb{R}$:
    \begin{enumerate}\itemsep5pt
        \item\label{ex-3-i}
        $\psi(\xi) = |\xi|^{\alpha}$ for $\alpha \in (1,2]$; \hfill (isotropic stable)
		
        \item\label{ex-3-ii}
        $\psi(\xi) = (|\xi|^2+m^2)^{\alpha/2}-m^{\alpha}$ for $\alpha \in (1,2)$, $m>0$; \hfill (relativistic stable)
		
        \item\label{ex-3-iii}
        $\psi(\xi) = - (|\xi|^2+m^2)^{\alpha/2} \cos \left( \alpha \arctan \frac{|\xi|}{m} \right)+ m^{\alpha}$ for $\alpha \in (1,2)$, $m>0$; \hfill (tempered stable)
		
        \item\label{ex-3-iv}
        $\psi(\xi) = (|\xi|^2+m)_{\alpha}-(m)_{\alpha}$ for some $\alpha \in (1,2)$, $m>0$; \hfill (Lamperti stable)
        \linebreak
        here $(t)_{\alpha} := \Gamma(t+\alpha)/\Gamma(t)$ denotes the Pochhammer symbol.
	\end{enumerate}

    If $b: (0,\infty) \times \mbb{R}^d \to \mbb{R}^d$ is a bounded function which is $\beta$-H\"{o}lder continuous with respect to $x$ and $\eta$-H\"{o}lder continuous with respect to $t$ for some $\eta \in (0,1]$ and $\beta>\frac{2}{\alpha}-1$, then the SDE
    \begin{equation*}
		dX_t = b(t,X_{t-}) \, dt + dL_t, \quad X_0 = x\in\mbb{R}^d,
	\end{equation*}
	has a pathwise unique strong solution. For any $p <\gamma_{\infty}$ and $T>0$ there exists a constant $C>0$ such that
    \begin{equation*}
		\mbb{E} \left( \sup_{t \leq T} |X_t-X_t^{(n)}|^p \right) \leq C {n^{-\min\{1,\,p \beta/\alpha,\,p \eta\}}} \fa n \geq 1
	\end{equation*}
    where we set $\gamma_{\infty} := \alpha$ for the exponent \ref{ex-3-i} and $\gamma_{\infty}:=\infty$ for all other exponents \ref{ex-3-ii}--\ref{ex-3-iv}.
\end{example}

\begin{remark} \label{ex-4}
\firstpara{(i)}
    The L\'evy measure of a tempered stable L\'evy process is given by
    \begin{equation*}
    	\nu(dy) =\frac{1}{2} \frac{\alpha (\alpha-1)}{\Gamma(2-\alpha)} e^{-m |y|} |y|^{-d-\alpha} \, dy
        \quad\text{for $\alpha\in (1,2]$}
    \end{equation*}
    cf.~\cite{kuechler} or \cite[Example 5.7]{matters}. Note that different authors use different names for this process, e.g.\ KoBoL process, CGMY process and truncated L\'evy process.

\para{(ii)}
    Example~\ref{ex-3} can be also shown by combining Theorem~\ref{main1} with the heat kernel estimates established in \cite{matters}, see also \cite{parametrix}; in fact, any continuous negative definite function listed in \cite[Table 2]{matters} satisfies the assumptions of Theorem~\ref{main1}.
\end{remark}

We close this section with a further example; it covers many interesting and important L\'evy processes.
\begin{example} \label{ex-5}
    Let $(L_t)_{t \geq 0}$ be a $d$-dimensional L\'evy process with characteristic exponent $\psi$ and L\'evy triplet $(0,0,\nu)$. Assume that $\nu$ is of the form \begin{equation}\label{ex-eq11}
		\nu(A) = \int_{\mbb{S}^{d-1}}\!\!\int_{(0,\infty)} \I_A(r \vartheta) Q(r) \, dr \,  \mu(d\vartheta),
        \quad A \in \mc{B}(\mbb{R}^d \setminus \{0\}) 	
    \end{equation}
	for a finite measure $\mu$ on the unit sphere $\mbb{S}^{d-1}$ in $\mbb{R}^d$ such that the support of $\mu$ is not contained in $\mbb{S}^{d-1} \cap V$ where $V \subseteq \mbb{R}^d$ is a lower-dimensional subspace, and a function $Q: (0,\infty) \to (0,\infty)$ satisfying
    \begin{equation*}
		0 < \liminf_{r \to 0} \frac{Q(r)}{r^{1+\gamma_0}} \leq \limsup_{r \to 0} \frac{Q(r)}{r^{1+\gamma_0}} < \infty
    \quad\text{and}\quad
		 \limsup_{r \to \infty} \frac{Q(r)}{r^{1+\gamma_{\infty}}} < \infty
	\end{equation*}
    for some $\gamma_0 \in (1,2]$ and $\gamma_{\infty}>0$. 	If $b: (0,\infty) \times \mbb{R}^d \to \mbb{R}^d$ is a bounded function which is $\beta$-H\"{o}lder continuous with respect to $x$ and $\eta$-H\"{o}lder continuous with respect to $t$ for some $\eta \in (0,1]$ and
    \begin{equation*}
		\gamma_0 (1+\beta)>2,
	\end{equation*}
	then the SDE
    \begin{equation*}
	   dX_t = b(t,X_{t-}) \, dt + dL_t, \quad X_0 = x\in\mbb{R}^d,
	\end{equation*}
	has a pathwise unique strong solution. For any $p <\gamma_{\infty}$ and $T>0$ there exists a constant $C>0$ such that
    \begin{equation*}
	   \mbb{E} \left( \sup_{t \leq T} |X_t-X_t^{(n)}|^p \right) \leq C {n^{-\min\{1,\,p \beta/\gamma_0,\,p \eta\}}} \fa n \geq 1.
	\end{equation*}
\end{example}

Since each of the processes in Example~\ref{ex-3} has a L\'evy measure of the form \eqref{ex-eq11}, Example~\ref{ex-5} is more general than Example~\ref{ex-3}. Let us point out that Example~\ref{ex-5} includes truncated stable L\'evy processes, i.e.\ $Q(r) = r^{-1-\alpha} \I_{(0,1)}(r)$, and layered stable L\'evy processes, i.e.\ $Q(r) = r^{-1-\alpha} \I_{(0,1)}(r) + r^{-1-\beta} \I_{[1,\infty)}(r)$.

\begin{proof}[Proof of Example~\ref{ex-5}]
    Some elementary calculations show that $c^{-1}|\xi|^{\gamma_0} \leq \re \psi(\xi) \leq c |\xi|^{\gamma_{0}}$ for some constant $c\in (0,\infty)$ as $|\xi| \to \infty$. Moreover, by the very definition of $\nu$, $\int_{|y| \leq 1} |y|^{\gamma_0+\eps} \, \nu(dy)+ \int_{|y| \geq 1} |y|^{\gamma_{\infty}-\eps} \, \nu(dy)<\infty$ for any $\eps>0$. Applying Corollary~\ref{main3} with $f(r) := r^{\gamma_0}$ finishes the proof.
\end{proof}

\appendix
\section{}

For the proof of our main results we use the following auxiliary statements.
\begin{proposition}[differentiation lemma for parameter-dependent integrals]\label{app-1}
    Let $(X,\mc{A},\mu)$ be a $\sigma$-finite measure space and $\phi: (a,b) \times X \to \mbb{R}$ a measurable function with the following properties. \begin{enumerate}
		\item\label{app-1-i} $\int_X |\phi(s,x)| \, \mu(dx)<\infty$ for all $s \in (a,b)$.
		\item\label{app-1-ii} $s \mapsto \phi(s,x)$ is differentiable for all $x \in X$ and 
        \begin{equation*}
		      \int_{(a,b)} \int_X |\partial_s \phi(s,x)| \, \mu(dx) \, ds < \infty.
		\end{equation*}
		\item\label{app-1-iii} $s \mapsto \int_X \partial_s \phi(s,x) \, \mu(dx)$ is continuous.
	\end{enumerate}
	Then $F(s) := \int_X \phi(s,x) \, \mu(dx)$ is continuously differentiable for all $s \in (a,b)$ and
    \begin{equation*}
		F'(s) = \int_X \partial_s \phi(s,x) \, \mu(dx), \quad s \in (a,b).
	\end{equation*}
\end{proposition}

Note that \ref{app-1-iii} is always satisfied if $t \mapsto \partial_t \phi(t,x)$ is continuous and there exists a function $w \in L^1(\mu)$ such that $|\partial_t \phi(t,x)| \leq w(x)$ for all $t \in (a,b)$ and $x \in X$; therefore, Proposition~\ref{app-1} extends the standard version of the differentiation lemma which can be found, for instance, in \cite[Theorem 12.5]{mims2}.

\begin{proof}[Proof of Proposition~\ref{app-1}]
	Fix $s \in (a,b)$. Applying the fundamental theorem of calculus and Fubini's theorem, we find
    \begin{align*}
		F(s+h)-F(s)
		= \int_X (\phi(s+h,x)-\phi(s,x)) \, \mu(dx)
		&= \int_X \int_{s}^{s+h} \partial_r \phi(r,x) \, dr \, \mu(dx) \\
		&= \int_s^{s+h} \int_X \partial_r \phi(r,x) \, \mu(dx) \, dr
	\end{align*}
	for all $h \in \mbb{R}$. By assumption, $f(r) := \int_X \partial_s \phi(r,x) \, \mu(dx)$ is continuous, and so
    \begin{equation*}
        \lim_{h \to 0} \frac{1}{h} (F(s+h)-F(s))
        = \lim_{h \to 0} \frac{1}{h} \int_s^{s+h} f(r) \, dr
        = f(s)
        \stackrel{\text{def}}{=} \int_X \partial_s \phi(s,x) \, \mu(dx).
    \qedhere
	\end{equation*}
\end{proof}

\begin{proposition} \label{app-3}
    Let $(L_t)_{t \geq 0}$ be a $k$-dimensional L\'evy process with L\'evy triplet $(\ell,0,\nu)$ and jump measure $N$ such that $\int_{|y| \geq 1} |y|^{\gamma}\, \nu(dy)<\infty$ holds for some $\gamma \in [1,2]$. Denote by $\mathcal F := (\mathcal F_t)_{t\geq 0}$, $\mathcal{F_t}:=\sigma(L_s; s\leq t)$, the natural filtration and let $b: [0,\infty) \times \Omega \to \mbb{R}^d$ and $\sigma: [0,\infty) \times \Omega \to \mbb{R}^{d \times k}$ be $\mathcal{F}$-progressively measurable bounded functions. Then the process
    \begin{equation*}
		X_t  := x + \int_0^t b(s) \, ds + \int_0^t \sigma(s) \, dL_s, \quad  x\in\mbb{R}^d
	\end{equation*}
	satisfies It\^o's formula  \marginpar{\footnotesize Please see the \newline correction on \\ the final page}
    \begin{align}\label{app-eq5}\begin{aligned}
		& F(t,X_t) -F(0,X_0) \\
		&= \int_{(0,t)} \partial_s F(s,X_s) \, ds +\int_{(0,t)} \nabla_x F(s,X_s) \cdot (b(s)+\sigma(s) \cdot \ell) \,ds \\
		&\quad\mbox{} + \iint_{(0,t)\times B(0,1)} (F(s,X_{s-}+\sigma(s) \cdot y)-F(s,X_{s-})) \, \tilde{N}(dy,ds) \\
		&\quad\mbox{} + \iint_{(0,t)\times B(0,1)^c} (F(s,X_{s-}+\sigma(s) \cdot y)-F(s,X_{s-})) \, N(dy,ds) \\
		&\quad\mbox{} + \iint_{(0,t)\times \mbb{R}^d} (F(s,X_{s-}+\sigma(s) \cdot y)-F(s,X_{s-})- \nabla_x F(s,X_{s-}) \cdot \sigma(s) y \I_{(0,1)}(|y|)) \, \nu(dy) \, ds \end{aligned}
	\end{align}
	for any function $F \in C_b^{1,1}((0,\infty) \times \mbb{R}^d)$ such that for every $T>0$
    \begin{equation*}
        \|\nabla_x F\|_{\mc{C}^{\gamma-1}([0,T])}
        := \sup_{t \in [0,T]}  \sup_{x \in \mbb{R}^d} |\nabla_x F(t,x)|
           +  \sup_{t \in [0,T]} \sup_{\substack{x,y \in \mbb{R}^d \\ x \neq y}} \frac{|\nabla_x F(t,x)-\nabla_y F(t,y)|}{|x-y|^{\gamma-1}} < \infty.
	\end{equation*}
\end{proposition}

\begin{proof}[Sketch of the proof]
    Fix $T>0$ and pick $\chi \in C_c^{\infty}(\mbb{R}^d)$ such that $\chi \geq 0$, $\int \chi(y) \, dy =1$ and $\chi(y)=0$ for all $|y| \geq 1$. If we set \begin{equation*}
		F_k(t,x) :=  k^d \int_{\mbb{R}^d} F(t,x+y) \chi(ky) \, dy, \quad x \in \mbb{R}^d,\; t \geq 0,\; k\in\mbb{N},
	\end{equation*}
	then $F_k \in C_b^{1,2}((0,\infty) \times \mbb{R}^d)$,
    \begin{equation*}
		\sup_{t \in [0,T]} \sup_{x \in K} \big(|(F-F_k)(t,x)| + |\nabla_x (F-F_k)(t,x)| + |\partial_t (F-F_k)(t,x)| \big)
        \xrightarrow[]{k \to \infty} 0
	\end{equation*}
	for any compact set $K \subseteq \mbb{R}^d$; moreover, we have
    \begin{equation*}
		\sup_{k \in\mbb{N}} \| \nabla_x F_k\|_{\mc{C}^{\gamma-1}([0,T])}
        \leq \|\nabla_x F\|_{\mc{C}^{\gamma-1}([0,T])}
        < \infty.
	\end{equation*}
	Applying Taylor's formula we find that there exist a sequence $c_k \to 0$ and a constant $C>0$ such that
    \begin{align} \label{app-eq7} \begin{gathered}
		|(F-F_k)(s,x+z)-(F-F_k)(s,x)| \leq c_k \min\{1,|z|\} \\
		|(F-F_k)(s,x+z)-(F-F_k)(s,x)-\nabla_x (F-F_k)(s,x) \cdot z| \leq C \min\{1, |z|^{\gamma}\}
    \end{gathered}\end{align}
    for all $x \in K$, $z \in \mbb{R}^d$ and $s \in [0,T]$.  Since we can apply It\^o's formula for each $F_k \in C_b^{1,2}((0,\infty) \times \mbb{R}^d)$, see e.g.\ \cite[Chapter II.5]{ikeda}, we get \eqref{app-eq5} for $F=F_k$. Define
    \begin{equation*}
		\tau_R := \inf\{t \geq 0; |X_t| \geq R\}.
	\end{equation*}
    Using \eqref{app-eq5} for $F_k$ and replacing $t$ by $t \wedge \tau_R$, it is not difficult to see that each of the integrals converges as $k\to\infty$: for the third integral (which is an $L^2$-martingale), we use It\^o's isometry, \eqref{app-eq7} and the dominated convergence theorem to get $L^2$-convergence and then we extract an almost surely convergent subsequence; all other integral expressions converge almost surely because of \eqref{app-eq7} and dominated convergence. This gives \eqref{app-eq5} for $F$ and with $t$ replaced by $t \wedge \tau_R$. Almost the same argument allows us now to let $R \to \infty$, and the claim follows.
\end{proof}

\begin{lemma} \label{app-5}
	Let $(X_t^{(n)})_{t \in [0,T]}$ be a sequence of stochastic processes with c\`adl\`ag sample paths such that
    \begin{equation*}
		\mbb{E} \left( \sup_{0 \leq t \leq T} \big|X_t^{(n)}-X_t^{(m)}\big|^p \right) \xrightarrow[]{m,n \to \infty} 0
	\end{equation*}
    for some $p>0$. Then there exists a stochastic process $(X_t)_{t \in [0,T]}$ and a subsequence $(n_k)_{k \in \mbb{N}}$ such that $\sup_{t \in [0,T]} |X_t-X_t^{(n_k)}| \to 0$ almost surely as $k \to \infty$.
\end{lemma}
\begin{proof}
    For $p \geq 1$ this follows from the Riesz--Fischer theorem on the completeness of the spaces $L^p$, see e.g.\ \cite[Theorem 13.7]{mims2}; therefore it suffices to consider the case $p \in (0,1)$. For $k \geq 1$ choose iteratively $n_k > n_{k-1}$ such that
    \begin{equation*}
		\mbb{E} \left( \sup_{0 \leq t \leq T} \big|X_t^{(n_k)} -X_t^{(n_{k-1})}\big|^p \right) \leq \frac{1}{2^k}.
	\end{equation*}
	As $p \in (0,1)$, we have $(x+y) \leq x^p + y^p$ for $x,y \geq 0$, and this implies
    \begin{equation*}
		\mbb{E} \left( \sum_{k \geq 1} \sup_{0 \leq t \leq T} \big|X_t^{(n_k)}-X_t^{(n_{k-1})}\big|^p \right)
		\leq \sum_{k \geq 1} \mbb{E} \left( \sup_{0 \leq t \leq T} \big|X_t^{(n_k)}-X_t^{(n_{k-1})}\big|^p \right)
		\leq \sum_{k \geq 1} \frac{1}{2^k} < \infty.
	\end{equation*}
    Since $X_t^{(n_i)} = X_t^{(n_0)} + \sum_{k=1}^i (X_t^{(n_k)}-X_t^{(n_{k-1})})$ this shows that the limit $X_t := \lim_{i \to \infty} X_t^{(n_i)}$ exists uniformly in $t \in [0,T]$ with probability $1$.
\end{proof}

\newpage
\noindent
\textbf{Correction to the present paper:} 

\bigskip\bigskip\noindent
In the statement of It\^o's formula, formula (38) of Proposition A.2, the last integral appearing on the right-side ranges only over $(0,t)\times B(0,1)$ rather than $(0,t)\times \mathbb{R}^d$, i.e.
\begin{align}\tag{38}
    \dots + \iint_{(0,t)\times\mathbb{R}^d} \dots\,\nu(dy)\,ds
    \xrightarrow[]{\text{should read}}
    \dots + \iint_{(0,t)\times B(0,1)} \dots\,\nu(dy)\,ds
\end{align}

We are using (38) on the proof of Theorem 2.1, p.~12, formula (28). The above mentioned change gives one further term on the right-hand side of (28)
\begin{align}
\tag{28}
    \dots\mbox{}-\int_{T_{i-1}}^t \int_{|y| \geq 1} (u_i(s,X_{s-}^{(n)}+y)-u_i(s,X_{s-}^{(n)})) \, \nu(dy) \, ds
\end{align} 
leading to an additional term $I_6$ on p.~13 (line 6 from above), i.e. 
\begin{align*}
    |X_t^{(m)}-X_t^{(n)}|^p \leq \ldots + C(I_1+I_2+I_3+I_{4,1}+I_{4,2}+I_5+ \fbox{$I_6$})
\end{align*} 
which is of the form
\begin{align*}
    I_6 := \left|\int_{T_{i-1}}^t \int_{|y| \geq 1} H_i(s,y) \, \nu(dy) \, ds \right|^p
\end{align*}
with $H_i$ defined on p.~13 (line 16 from below). Since $\nu(B(0,1)^c)<\infty$, we can estimate this term using (29): 
\begin{align*}
    \mathbb{E} \left( \sup_{T_{i-1} \leq t \leq T_i} |I_6| \right) 
    \leq 2^p \epsilon^p T^p \mbb{E} \left( \sup_{T_{i-1} \leq s \leq T_i} |X_s^{(m)}-X_s^{(n)}|^p \right) \left( \int_{|y| \geq 1} \, \nu(dy)\right)^p,
\end{align*}
as is needed for the rest of the proof. 

\bigskip\noindent
\emph{We are grateful to Changsong Deng \textup{(}Wuhan University\textup{)} for pointing out the error in formula \textup{(38)}.}


\begin{thebibliography}{99}\frenchspacing	
\bibitem{chen}
	Chen, Z.-Q., Song, R., Zhang, X.: Stochastic flows for L\'evy processes with H\"{o}lder drifts. Preprint arXiv 1501.04758.
	
\bibitem{chen2}
	Chen, Z.-Q., Zhang, X., Zhao, G.: Well-posedness of supercritical SDEs driven by L\'evy processes with irregular drifts. Preprint arXiv 1709.04632.
	
\bibitem{deng}
    Deng, C.-S., Schilling, R.\,L.: On shift Harnack inequalities for subordinate semigroups and moment estimates for L\'evy processes. \emph{Stoch.\ Proc.\ Appl.} \textbf{125} (2015), 3851--3878.

\bibitem{proske}
    Haadem, S., Proske, F.: On the Construction and Malliavin Differentiability of L\'evy Noise Driven SDEs with Singular Coefficients. \emph{J.\ Funct.\ Anal.} \textbf{266} (2014), 5321--5359.

\bibitem{hashimoto}
    Hashimoto, H.: Approximation and stability of solutions of SDEs driven by a symmetric $\alpha$ stable process with non-Lipschitz coefficients. In: Donati, M.\,C., Leiay, A., Rouault, A. (eds): \emph{S\'eminaire de Probabilit\'es XLV}. Lecture Notes in Mathematics \textbf{2078}. Springer, 2013.

\bibitem{kloeden}
    Higham, D.\,J, Kloeden, P.\,E.: Strong convergence rates for backward Euler on a class of non-linear jump-diffusion problems. \emph{J.\ Computat.\ Appl.\ Math.} \textbf{205} (2007), 949--956.

\bibitem{ikeda}
	Ikeda, N., Watanabe, S.: \emph{Stochastic Differential Equations and Diffusion Processes}. North-Holland, Amsterdam 1989 (2nd edn).

\bibitem{jacod}
	Jacod, J.: The Euler scheme for L\'evy driven stochastic differential equations: Limit theorems. \emph{Ann.\ Probab.} \textbf{32} (2004), 1830--1872.
	
\bibitem{knop}
    Knopova, V., Schilling, R.\,L.: A note on the existence of transition probability densities for L\'evy processes. \emph{Forum Math.} \textbf{25} (2013), 125--149.

\bibitem{kuechler}
    K\"uchler, U., Tappe, S.: Tempered stable distributions and processes. \emph{Stoch.\ Proc.\ Appl.} \textbf{123} (2013), 4256--4293.

\bibitem{moments}
    K\"{u}hn, F.: Existence and estimates of moments for L\'evy-type processes. \emph{Stoch.\ Proc.\ Appl.} \textbf{127} (2017), 1018--1041.

\bibitem{parametrix}
    K\"{u}hn, F.: Transition probabilities of L\'evy-type processes: Parametrix construction. Preprint arXiv 1702.00778.

\bibitem{matters}
    K\"{u}hn, F.: \emph{L\'evy-Type Processes: Moments, Construction and Heat Kernel Estimates}. Springer Lecture Notes in Mathematics \textbf{2187} (vol.~VI of the ``L\'evy Matters'' subseries). Springer, 2017.

\bibitem{ihke}
    K\"{u}hn, F., Schilling, R.\,L.: On the domain of fractional Laplacians and related generators of Feller processes. Preprint arXiv 1610.08197.

\bibitem{schoenberg}
	K\"{u}hn, F., Schilling, R.\,L.: A probabilistic proof of Schoenberg's theorem. \emph{In preparation}.

\bibitem{kulik}
	Kulik, A.: On weak uniqueness and distributional properties of a solution to an SDE with $\alpha$-stable noise. To appear: Stoch.\ Proc.\ Appl. DOI 10.1016/j.spa.2018.03.010.

\bibitem{lunardi}
    Lunardi, A.: \emph{Interpolation Theory}. Edizioni Della Normale, Pisa 2009 (2nd edn).

\bibitem{matheron}
    Matheron, G.: \emph{Les variables r\'egoionalis\'ees et leur estimation. Une application de la th\'eorie des fonctions al\'eatoires aux sciences de la nature}. Masson \& Cie., Paris 1965.

\bibitem{mik16}
	Mikulevicus, R., Xu, F.: On the rate of convergence of strong Euler approximation for SDEs driven by L\'evy processes. \emph{Stochastics} \textbf{90} (2018), 569--604.

\bibitem{novikov}
	Novikov, A.\,A.: On discontinuous martingales. \emph{Theory Probab.\ Appl.} \textbf{20} (1975), 11--26.

\bibitem{pamen}
    Pamen, O.\,M., Taguchi, D.: Strong rate of convergence for the Euler--Maruyama approximation of SDE with H\"{o}lder continuous drift coefficient. \emph{Stoch.\ Proc.\ Appl.} \textbf{127} (2017), 2542--2559.

\bibitem{protter}
    Protter, P.: \emph{Stochastic Integration and Differential Equations}. Springer, Berlin 2004 (2nd edn).

\bibitem{qiao}
	Qiao, H.: Euler--Maruyama approximation for SDEs with jumps and non-Lipschitz coefficients. \emph{Osaka J.\ Math.} \textbf{51} (2014), 47--66.

\bibitem{sato}
	Sato, K.: \emph{L\'evy Processes and Infinitely Divisible Distributions}. Cambridge University Press, Cambridge 2005.

\bibitem{mims2}
	Schilling, R.\,L.: \emph{Measures, Integrals and Martingales}. Cambridge University Press, Cambridge 2017 (2nd edn).

\bibitem{bernstein}
	Schilling, R.\,L., Song, R., Vondra\v{c}ek, Z.: \emph{Bernstein functions}. De Gruyter, Berlin 2012 (2nd edn).

\bibitem{ssw}
    Schilling, R.\,L., Sztonyk, P., Wang, J.: Coupling property and gradient estimates of L\'evy processes via the symbol. \emph{Bernoulli} \textbf{18} (2012), 1128--1149.

\bibitem{tanaka}
	Tanaka, H., Tsuchiya, M., Watanabe, S.: Perturbation of drift-type for L\'evy processes. \emph{J.\ Math.\ Kyoto Univ.} \textbf{14} (1974), 73--92.

\bibitem{triebel}
    Triebel, H.: \emph{Interpolation Theory, Function Spaces, Differential Operators}. Johann Ambrosius Barth, Heidelberg 1995 (2nd edn).
\end{thebibliography}
\end{document}